\documentclass[reqno, 10pt]{article}

\usepackage{fullpage}

\usepackage[centertags]{amsmath}
\usepackage{amsmath}
\usepackage{amsfonts}
\usepackage{amssymb}
\usepackage{amsthm}
\usepackage{enumerate}

\usepackage{newlfont}
\usepackage{color}

\usepackage{authblk}

\usepackage{stmaryrd}
\SetSymbolFont{stmry}{bold}{U}{stmry}{m}{n}

\usepackage{tikz}

\usepackage{bbm}
\include{amsthm_sc}

\usepackage{stmaryrd}
\usepackage{mathrsfs}

\usepackage{fancyhdr}
\usepackage{graphicx}
\usepackage{fancybox}
\usepackage{setspace}
\usepackage{cleveref}

\newtheorem{thm}{Theorem}
\newtheorem{cor}[thm]{Corollary}
\newtheorem{lem}[thm]{Lemma}
\newtheorem{prop}[thm]{Proposition}

\theoremstyle{definition}
\newtheorem{defn}[thm]{Definition}

\theoremstyle{remark}
\newtheorem{rem}[thm]{Remark}
\newtheorem{nrem}[thm]{Notational Remark}

\newcommand{\R} {\mathbb{R}}
\newcommand{\C} {\mathbb{C}}

\newcommand{\E} {\mathbb{E}}

\newcommand{\p} {\mathbb{P}}

\DeclareMathOperator{\Tr}{Tr}
\DeclareMathOperator{\var}{var}

\DeclareMathOperator{\re}{\mathrm{Re}}
\DeclareMathOperator{\im}{\mathrm{Im}}

\DeclareMathOperator{\SSK}{SSK}
\DeclareMathOperator{\CW}{CW}
\DeclareMathOperator{\TW}{TW}
\DeclareMathOperator{\exte}{ext}
\DeclareMathOperator{\covar}{cov}

\newcommand{\caF}{{\mathcal F}}

\newcommand{\caK}{{\mathcal K}}

\newcommand{\caN}{{\mathcal N}}
\newcommand{\caO}{{\mathcal O}}

\newcommand{\bss}{{\boldsymbol s}}

\newcommand{\wt}{\widetilde}
\newcommand{\ol}{\overline}

\newcommand{\beq}{ \begin{equation} }
\newcommand{\eeq}{ \end{equation} }

\newcommand{\lone}{\mathbbm{1}} 

\newcommand{\dd}{\mathrm{d}}
\newcommand{\ii}{\mathrm{i}}

\renewcommand{\bss}{\boldsymbol{\sigma}}

\newcommand{\Aq}{S}
\newcommand{\Aw}{T}

\numberwithin{equation}{section} 
\numberwithin{thm}{section}

\title{Fluctuations of the free energy of the spherical Sherrington--Kirkpatrick model with ferromagnetic interaction}

\author{Jinho Baik\footnote{Department of Mathematics, University of Michigan,
Ann Arbor, MI, 48109, USA \newline email: \texttt{baik@umich.edu}}
and Ji Oon Lee\footnote{Department of Mathematical Sciences, KAIST, Daejeon, 34141, Korea
\newline email: \texttt{jioon.lee@kaist.edu}}}

\begin{document}

\maketitle

\begin{abstract}
We consider a spherical spin system with pure 2-spin spherical Sherrington--Kirkpatrick Hamiltonian with ferromagnetic Curie--Weiss interaction. 
The system shows a two-dimensional phase transition with respect to the temperature and the coupling constant. 
We compute the limiting distributions of the free energy for all parameters away from the critical values.
The zero temperature case corresponds to the well-known phase transition of the largest eigenvalue of a rank 1 spiked random symmetric matrix. 
As an intermediate step, we establish a central limit theorem for the linear statistics of rank 1 spiked random symmetric matrices. 
\end{abstract}

\section{Introduction}

\subsection{Model}

Let $A=(A_{ij})_{i,j=1}^N$ be a real symmetric matrix where $A_{ij}$, $1 \leq i < j \leq N$, are independent random variables with mean $0$ and variance $1$, and 
the diagonal entries $A_{ii}=0$. 
The pure $2$-spin spherical Sherrington--Kirkpatrick (SSK) model with no external field is a disordered system defined by the 
random Hamiltonian 
\beq
	H_N^{\SSK} (\bss) := \frac1{\sqrt{N}} \langle \bss, A \bss \rangle 
	= \frac1{\sqrt{N}} \sum_{i,j=1}^N A_{ij} \sigma_i\sigma_j
\eeq
for the spin variables on the sphere,  $\bss \in S_{N-1}$, where $S_{N-1} := \{ \bss \in \R^N : \| \bss \|^2 = N \}$.
For the history and the existing results on the model, including the proof of the Parisi formula, we refer to \cite{KosterlitzThoulessJones, CrisantiSommers, TalagrandParisiSpher, PanchekoTalagrand2007} and references therein.

We are interested in the spherical spin system with (random) Hamiltonian
\beq \label{eq:defofH}
	H_N (\bss) = H_N^{\SSK}(\bss) + H_N^{\CW}(\bss), \qquad \bss \in S_{N-1}, 
\eeq
where the Curie--Weiss (CW) Hamiltonian with coupling constant $J$ is defined by 
\beq
	H_N^{\CW}(\bss) := \frac{J}{N}  \sum_{i,j=1}^N \sigma_i\sigma_j 
	=  \frac{J}{N} \left( \sum_{i=1}^N \sigma_i \right)^2. 
\eeq
Note that $H_N^{\CW}$ is large in magnitude when all $\sigma_i$ have the same sign. 
The Hamiltonian $H_N$ is similar to the SSK model with external field,  
\begin{equation}
	H_N^{\exte}(\bss)= H_N^{\SSK}(\bss)+ h \sum_{i=1}^N \sigma_i.
\end{equation} 
See \cite{Chen2014} for a relation between these two Hamiltonians. 

The main result of this paper is a limit theorem for the free energy at positive temperature $1/\beta>0$ with positive coupling constant $J$. 
This paper is an extension of our previous paper \cite{BaikLee} in which we obtained limit theorems for pure $2$-spin SSK model (with $J=0$). 

\medskip

Before we state our result, we first summarize the known limit theorems for the free energy of SSK and also the Sherrington--Kirkpatrick (SK) model.
Results for (3)--(7) were 
established in the same year 2015. 
We indicate the limiting distribution and the order of fluctuations of the free energy. 
These results assume that $A_{ij}$ are standard Gaussian. 
However, the result (1) was extended to non-Gaussian $A_{ij}$ in \cite{GuerraToninelli2002, CarmonaHu}, 
and the results (3) and (4) were obtained for general normalized random variables. 

\subsubsection*{No external field.}
When there is no external field ($h=0$), the following are known for pure $p$-spin models: 
\begin{enumerate}[(1)]
	\item Pure $2$-spin SK model for $\beta\in (0, \beta_c)$: Gaussian, $O(N^{-1})$ \cite{AizenmanLebowitzRuelle, FrohlochZegarlinski1987, CometsNeveu1995}
	\item Pure $p$-spin SK model for $\beta\in (0, \beta^{(p)}_{c})$: Gaussian, $O(N^{-p/2})$ \cite{BovierKurkovaLowes}
	\item Pure $2$-spin SSK model for $\beta\in (0, \beta_c)$: Gaussian,  $O(N^{-1})$ \cite{BaikLee}
	\item Pure $2$-spin SSK model for $\beta\in (\beta_c, \infty)$: $\TW_1$,  $O(N^{-2/3})$ \cite{BaikLee}
	\item Pure $p$-spin SSK model for $p\ge 3$  at $\beta=\infty$: Gumbel,  $O(N^{-1})$ \cite{SubagZeitouni2015}
\end{enumerate}
Here, $\TW_1$ denotes the GOE Tracy-Widom distribution. 
The numbers $\beta^{(p)}_c$, $p\ge 3$, and $\beta_c$ are certain critical values. 
A results for pure $p$-spin SSK model with $p\ge 3$ for low temperature is given in \cite{Subag2016}.

We also remark that the free energy of the pure $2$-spin SSK model at zero temperature, $\beta=\infty$, is, after modifying the definition slightly, equal to the rescaled largest eigenvalue of symmetric random matrix $A$. 
Hence, from the well-known result in the random matrix theory \cite{So1999, TV2010, EYY}, this case also corresponds to $\TW_1$ with $O(N^{-2/3})$ fluctuation. 
Comparing with (5), we find that at zero temperature the free energy fluctuates differently for $p=2$ and $p\ge 3$. 
There is an important difference of $p=2$ case and $p\ge 3$ case: the number of critical points for the Hamiltonian (subject to the constraint $\|\bss\|^2=N$) is $2N$ for $p=2$, but is exponential in $N$ for $p\ge 3$ as proved in \cite{ABC2013} (for upper bound) and \cite{Subag2015} (for lower bound). The critical points are the eigenvectors of $A$ for $p=2$, hence strongly correlated, whereas the extremal process of critical points converges in distribution to a Poisson point process for $p \ge 3$. See Theorem 1 of \cite{SubagZeitouni2015} for more detail.

\subsubsection*{Positive external field.}

The behavior of the free energy changes drastically under the presence of an external field ($h>0$). For this case, the more complicated model with mixed $p$-spin interactions were also studied. 
\begin{enumerate}
	\item[(6)] Mixed $p$-spin SK and SSK models (without odd $p$-interactions for $p\ge 3$) with $h>0$ for all $\beta\in (0, \infty)$: Gaussian, $O(N^{-1/2})$ \cite{ChenDeyPanchenkp2015}
	\item[(7)] Mixed $p$-spin SSK model with $h>0$ at $\beta=\infty$: Gaussian, $O(N^{-1/2})$ \cite{ChenSen2015} 
\end{enumerate}
Note that the fluctuations are significantly increased from the $h=0$ case. 

It is interesting to scale $h\to 0$ with $N$ and consider a transition from (6) and (7) to (4) or (5). 
By matching the variance when $h=0$ and $h>0$, it is expected that the transitional scaling is $h=O(N^{-1/6})$ for $p=2$.  
For the large deviation analysis for the pure $2$-spin SSK model and discussions for such $h$, we refer to \cite{FyodorovLeDoussal2014} for deterministic $h$ and \cite{DemboZeitouni2015} for random $h$.

\subsection{Definitions}

We first define a Hamiltonian that generalizes $H_N$ in \eqref{eq:defofH}.

\begin{defn}[Interactions] \label{def:M} \label{def:Wigner} \label{cond:nonzero}
Let $A_{ij}$, $1\le i\le j$, be independent real random variables satisfying the following conditions:
\begin{itemize}
\item All moments of $A_{ij}$ are finite and $\mathbb{E}[A_{ij}]=0$. 
\item For all $i<j$, $\mathbb{E}[A_{ij}^2]=1$, $\mathbb{E}[A_{ij}^3]=W_3$, and $\mathbb{E}[A_{ij}^4]=W_4$ for some constants $W_3\in \R$ and $W_4> 0$. 
\item For all $i$, $\mathbb{E}[A_{ii}^2]=w_2$ for a constant $w_2\ge 0$. 
\end{itemize}
Set $A_{ji}=A_{ij}$ for $i<j$, and set $A=(A_{ij})_{i,j=1}^N$. Let
\beq \label{defofMma}
M_{ij} = \frac{A_{ij}}{\sqrt N} + \frac{J}{N} \quad (i \neq j), \qquad M_{ii} = \frac{A_{ii}}{\sqrt N} + \frac{J'}{N}
\eeq
for some ($N$-independent) non-negative constants $J$ and $J'$. 
Set $M=(M_{ij})_{i,j=1}^N$. We call $M$ a Wigner matrix with non-zero mean. 
\end{defn}

The Hamiltonian in \eqref{eq:defofH} is obtained by setting $A_{ii}=0$ and $J'=0$.

\begin{defn}[Free energy] \label{def:partition}
Define the Hamiltonian $H_N (\bss) =  \langle \bss, M \bss \rangle$ on sphere $\|\bss\|=\sqrt{N}$.
For $\beta>0$, define the partition function and the free energy as 
\beq \label{eq:partition}
	Z_N= Z_N(\beta)= \int_{S_{N-1}} e^{\beta H_N(\bss)} \dd \omega_N(\bss), 
	\qquad 
	F_N= F_N(\beta)= \frac1{N} \log Z_N, 
\eeq
where $\dd\omega_N$ is the normalized uniform measure on the sphere $S_{N-1} = \{ \bss \in \R^N : \| \bss \|^2 = N \}$.
\end{defn}

\begin{rem}
We may also consider complex matrix $M$. In this case, the real and the complex entries are independent and we add an extra condition that $\E A_{ij}^2=0$. 
The results in this paper have corresponding results for complex $M$, but we do not state them here.
\end{rem}

\subsection{Results} \label{sec:resultsm}

The following is the main result. 
The case when $J=0$ was proved previously in \cite{BaikLee}.

\begin{thm} \label{thm:main}
The following holds as $N\to \infty$ where all the convergences are in distribution. 
The notation $\caN(a,b)$ denotes Gaussian distribution with mean $a$ and variance $b$ and $\TW_1$ is the GOE Tracy--Widom distribution. 
\begin{enumerate}[(i)]
\item (Spin glass regime) If $\beta > \frac{1}{2}$ and $J < 1$, then 
\beq\label{eq:thmspin1low}
	\frac{1}{\beta-\frac{1}{2}} N^{2/3} \left( F_N - F(\beta) \right) \Rightarrow \TW_{1}.
\eeq

\item (Paramagnetic regime) If $\beta < \frac{1}{2}$ and  $ \beta < \frac{1}{2J}$, then 
\beq\label{eq:thmspin1high}
	N \left( F_N - F(\beta) \right) \Rightarrow \mathcal{N} \left(f_1, \alpha_1 \right). 
\eeq

\item (Ferromagnetic regime) If $J > 1$ and $\beta > \frac{1}{2J}$, then 
\beq \label{eq:thmspin1mid}
	\sqrt{N} \left( F_N - F(\beta) \right) \Rightarrow \mathcal{N} \left(0, \alpha_2 \right). 
\eeq
\end{enumerate}
The leading order limit of the free energy is given by
\beq \label{Flimit}
	F(\beta)= \begin{cases}
	2\beta - \frac12 \log(2\beta) -\frac34  \qquad &\text{for (i)} \\
	\beta^2	\qquad &\text{for (ii)} \\
	\beta \left( J + \frac{1}{J} \right)  - \frac{1}{2} \log (2\beta J) - \frac{1}{4J^2} -  \frac{1}{2},  \quad &\text{for (iii).}
	\end{cases}
\eeq
The parameters for case (ii) in \eqref{eq:thmspin1high} are
\beq \label{eq:f1}
	f_1 = \frac{1}{4} \log(1-4\beta^2) + \beta^2(w_2-2) + 2\beta^4(W_4-3) - \beta J - \frac{1}{2} \log(1-2\beta J) + \beta J'
\eeq
and
\beq \label{eq:alpha1}
	\alpha_1 = -\frac{1}{2} \log (1-4\beta^2) + \beta^2 (w_2-2) + 2\beta^4 (W_4-3),
\eeq
and the parameter for case (iii) in \eqref{eq:thmspin1mid} is
\beq
	\alpha_2 = 2 \left( 1 - \frac{1}{J^2} \right) \left(\beta - \frac{1}{2J} \right)^2.
\eeq
\end{thm}

If we set $T = \frac{1}{2\beta}$, then the trichotomy corresponds to the cases $\max \{ T, J, 1 \} = 1$, $\max \{T, J, 1\} = T$, and $\max \{ T, J, 1 \} = J$,  respectively.
See Figure \ref{fig:diagram} for the phase diagram. 

\begin{figure}
\centering
\begin{tikzpicture}[scale=0.9]
\draw[thick, ->] (0,0) -- (4,0);
\draw[thick, ->] (0,0) -- (0,4);
\draw[thick] (2, 0) -- (2,2);
\draw[thick] (0, 2) -- (2,2);
\draw[thick] (2, 2) -- (3.6, 3.6);
\draw node at (4.3,0) {$J$};
\draw node at (-0.5, 3.6) {$\frac1{2\beta}$};
\draw node at (-0.2, -0.25) {$0$};
\draw node at (2, -0.25) {$1$};
\draw node at (-0.25, 2) {$1$};
\draw node at (1,1) {spin glass};
\draw node at (1.3,3) {paramagnetic};
\draw node at (3.5,1) {ferromagnetic};
\end{tikzpicture}
\caption{Phase diagram}
\label{fig:diagram}
\end{figure}
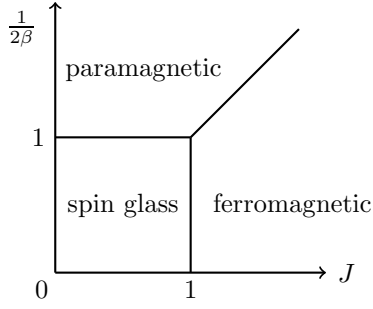

The above result implies that 
\beq
	F_N(\beta) \to F(\beta)
\eeq
in probability for $(J, \beta)$ not on the critical lines.
The formula \eqref{Flimit} of $F(\beta)$, and hence also the phase diagram, were obtained by Kosterlitz, Thouless, and Jones \cite{KosterlitzThoulessJones}.
Their proof is not completely rigorous but can be made rigorous by using the estimates that were developed later in random matrix theory. 
In this paper, we make their analysis rigorous and improve it to obtain the results on the fluctuations. 
Note that even though 
the paramagnetic regime and the ferromagnetic regime both have a Gaussian as the limiting distribution, 
the order of the fluctuations are different. 
The reason for this can be seen from the following theorem of which 
Theorem \ref{thm:main} is a consequence. 


\begin{thm} \label{thm:main2}
Let $\mu_1\ge \mu_2\ge\dots\ge \mu_N$ be the eigenvalues of Wigner matrix with non-zero mean $M$ in Definition \ref{def:M}. 
For every $\epsilon>0$ and $D > 0$, the following holds as $N\to \infty$ with probability higher than $1-N^{-D}$.
\begin{enumerate}[(i)]
\item (Spin glass regime) If $\beta > \frac{1}{2}$ and $J < 1$, then 
\beq\label{eq:thmspin1low002}
	F_N=  F(\beta) + \left( \beta-\frac{1}{2} \right) \left( \mu_1- 2 \right)
	+ O(N^{-1+\epsilon}).
\eeq

\item (Paramagnetic regime) If $\beta < \frac{1}{2}$ and $ \beta < \frac{1}{2J}$, then
\beq \label{eq:thmspin1high002} \begin{split}
	F_N 
= 2\beta^2   - \frac12 \log (2\beta)   - \frac{1}{2N} \sum_i g(\mu_i)
+ \frac1{N} \left( \log(2\beta) -\frac12  \log \left( -\frac1{N} \sum_i g''(\mu_i) \right) \right)  + O(N^{-2+\epsilon}) 
\end{split} \eeq
where
\beq
	g(x) = \log \left( 2\beta+\frac1{2\beta} - x \right).
\eeq

\item (Ferromagnetic regime) If $J > 1$ and $\beta > \frac{1}{2J}$, then 
\beq \label{eq:thmspin1mid002}
	F_N = F(\beta) + \left(\beta - \frac1{2J} \right) \left( \mu_1 - J- \frac1{J} \right) + O(N^{-1}\log N).
\eeq
\end{enumerate}
\end{thm}

Intuitively, 
the free energy is dominated by the ground state, $\mu_1$, at low temperature, and by all eigenvalues at high temperature. 
The above result makes this intuition precise: in the spin glass regime (i) and the ferromagnetic regime (iii), the fluctuations of the free energy are governed by the ground state, the largest eigenvalue $\mu_1$, while in the ferromagnetic regime (ii), they are governed by all of the eigenvalues in the form of the linear statistics $\sum_i g(\mu_i)$ of a specific function $g$.

The Wigner matrix with non-zero mean $M$ is a rank 1 case of so-called a spiked random matrix. A spiked random matrix is a random matrix perturbed additively by a deterministic matrix of fixed $N$-independent rank. 
Spiked random matrices were studied extensively in random matrix theory \cite{Baik-Ben_Arous-Peche05, Feral-Peche07, CDF2012, PRS2013, KY2013_iso}. 
Since the perturbation has a rank independent of $N$, the semi-circle law (see \eqref{semicirlw} below) still holds. 
However, the top eigenvalues may have different limit theorems. 
For the rank 1 case $M$, it was shown in Theorem 1.3 of \cite{PRS2013} that
\beq \label{spikedr1fl} 
	\begin{cases} N^{2/3} ( \mu_1 - 2 ) \Rightarrow \TW_1, \qquad &J<1 \\
	N^{1/2} \left(\mu_1 - (J + \frac{1}{J}) \right) \Rightarrow \caN(0, 2(1-\frac{1}{J^2})),\qquad &J>1 .
\end{cases}
\eeq
(See also Theorem 3.4 of \cite{CDF2012}.) For Hermitian matrix, \eqref{spikedr1fl} was first proved in \cite{Baik-Ben_Arous-Peche05}.
When $J<1$, then the perturbation has little effect on $\mu_1$. 
But when $J>1$, $\mu_1$ becomes an ``outlier'' in the sense that it is separated from the support of the semi-circle and as a consequence, becomes ``freer'' to fluctuate; the fluctuation order $N^{-1/2}$ is bigger in this case. 
Theorem \ref{thm:main} (i) and (iii) follow directly from Theorem \ref{thm:main2} and  \eqref{spikedr1fl}.

\subsection{Linear statistics for Wigner matrix with non-zero mean}\label{sec:linstsu}

In order to prove Theorem \ref{thm:main} (ii) from Theorem \ref{thm:main2} (ii), we need a limit theorem for the linear statistic $\sum_i g(\mu_i)$. 
It is a well-known result in random matrix theory that for mean-zero Wigner matrices (i.e. $J=0$ case), the linear statistics converge to Gaussian distributions with scale $O(1)$ instead of the classical diffusive $O(N^{1/2})$ scale for the sum of independent random variables \cite{Johansson98, SiSo, BS2004, BY2005, LP}.
The main technical component of this paper is 
the central limit theorem for the linear statistics of Wigner matrix with non-zero mean 
(i.e. $J>0$ case). 
The next theorem shows that the spike (i.e. $J>0$) only changes the mean of the limiting Gaussian distribution; the variance of the Gaussian distribution is same for all $J\ge 0$. 
We remark that the change of the mean due to the spike is already known for spiked sample covariance matrices \cite{WSY, PMC}.

We prove the following result for $J>0$. 
Set
\beq \label{Chebyshev formula}
	\tau_\ell(\varphi)= \frac1{\pi} \int_{-2}^2 \varphi(x) \frac{T_\ell(x/2)}{\sqrt{4-x^2}} \, \dd x = \frac1{2\pi} \int_{-\pi}^\pi \varphi(2\cos\theta) \cos(\ell\theta) \, \dd \theta
\eeq
for $\ell=0,1,2,\dots$, where $T_\ell(t)$ are the Chebyshev polynomials of the first kind; $T_0(t)=1$, $T_1(t)=t$, $T_2(t)=2t^2-1$, $T_3(t)=4t^3-3t$, $T_4(t)=8t^4-8t^2+1$, etc.

\begin{thm}[Linear statistics of Wigner matrix with non-zero mean]\label{thm:linear}
Let $M$ be an $N \times N$ Wigner matrix with non-zero mean as in Definition \ref{def:M}. Denote by $\mu_1 \geq \mu_2 \geq \dots \geq \mu_N$ the eigenvalues of $M$. Set
\beq \label{hat J}
\widehat J =
	\begin{cases}
	J + J^{-1} & \text{ if } J>1 \,, \\
	2 & \text{ if } J\leq 1 \,.
	\end{cases}
\eeq
Then, for any function $\varphi : \R \to \R$ that is analytic in an open neighborhood of $[-2, \widehat{J} \, ]$ and has compact support, the random variable
\beq
T_N(\varphi) := \sum_{i=1}^N \varphi(\mu_i) - N \int_{-2}^2 \varphi(x) \frac{\sqrt{4-x^2}}{2\pi} \dd x
\eeq
converges in distribution to the Gaussian distribution 
with mean $M(\varphi)$ and variance $V(\varphi)$, where
\beq \label{eq:Mvarfo}
\begin{split}
M(\varphi) &= \frac{1}{4} \left( \varphi(2) + \varphi(-2) \right) -\frac{1}{2} \tau_0(\varphi) + J' \tau_1(\varphi) + (w_2-2) \tau_2(\varphi) + \left( W_4 - 3 \right) \tau_4(\varphi) \\
&\qquad + \frac{1}{2\pi \ii} \oint  \varphi \left(-s-\frac{1}{s} \right)  \frac{J^2 s}{1 + Js}  \dd s
\end{split} \eeq
and
\beq \label{eq:Vvarfo}
V(\varphi) = (w_2 -2) \tau_1(\varphi)^2 + (W_4 -3) \tau_2(\varphi)^2 + 2 \sum_{\ell=1}^{\infty} \ell \tau_{\ell}(\varphi)^2.
\eeq
The contour for the integral in \eqref{eq:Mvarfo} is any simple closed contour containing $0$ inside in the 
slit disk $\{ |s|<1\}\setminus [-1, -1/J]$ in which $\varphi \left(-s-\frac{1}{s} \right)$ is analytic. (The analyticity condition of $\varphi$ implies that there is such a contour.) 
\end{thm}

Note that the variance does not depend on $J$ and $J'$ but the mean does.

Among various methods of studying the linear statistics in random matrix theory, we follow the method of  Bai and Silverstein, and Bai and Yao \cite{BS2004, BY2005} to prove the above result. 
Specifically, we extend the analysis of \cite{BY2005} to the $J>0$ case.
Let $\rho_N=\frac{1}{N} \sum_{j=1}^N \delta_{\mu_j}$ be the empirical spectral distribution of $M$.
As $N \to \infty$, $\rho_N$ converges to the semicircle measure $\rho$, defined by
\beq \label{semicirlw}
\rho(\dd x) = \frac{1}{2\pi} \sqrt{4-x^2}_+ \dd x.
\eeq
Let $s_N(z)$ and $s(z)$ be the Stieltjes transforms of $\rho_N$ and $\rho$, respectively, for $z \in \C^+$. Then, $T_N(\varphi)$ admits an integral representation, which can be easily converted to a contour integral that contains $\xi_N(z) := s_N(z)-s(z)$ in its integrand. The problem then reduces to showing that $\xi_N(z)$ converges to a Gaussian process $\xi(z)$. Due to the non-zero mean of the entries $M_{ij}$, the proof of convergence of $\xi_N(z)$ and the evaluation of the mean and the covariance of $\xi(z)$ become complicated. The main technical input we use in the estimate is the local semicircle law obtained in \cite{EKYY1}.

Theorem \ref{thm:main} (ii) follows from Theorem \ref{thm:main2} and Theorem \ref{thm:linear} once we evaluate the mean and the variance of the limiting Gaussian distribution: see Section \ref{sec:proofofthmma}. 

\begin{rem}
It is direct to check that 
the integral in \eqref{eq:Mvarfo} can also be expressed as: 
\beq \label{eq:Mvarfo22}
\begin{split}
	\frac{1}{2\pi \ii} \oint  \varphi \left(-s-\frac{1}{s} \right)  \frac{J^2 s}{1 + Js}  \dd s
	= \begin{cases}
	\sum_{\ell=2}^{\infty} J^{\ell} \tau_{\ell}(\varphi) & \text{ if } J<1,\\
	\frac{1}{2} \varphi(2) -\frac{1}{2} \tau_0(\varphi) - \tau_1(\varphi) & \text{ if } J=1, \\
	\varphi ( \widehat J ) - \tau_0(\varphi) - \widehat J \tau_1(\varphi) - \sum_{\ell=2}^{\infty} J^{-\ell} \tau_{\ell}(\varphi) & \text{ if } J>1.
	\end{cases}
\end{split} \eeq\end{rem}

\subsection{Transitions} \label{sec:transitions}

It is interesting to consider the phase transition and the near-critical behaviors in Theorem \ref{thm:main} and \ref{thm:main2}. 
We have the following result for the  transition between the spin glass regime (i) and the ferromagnetic regime (iii). 
For fixed $\beta > 1/2$, 
consider $J$ depending on $N$ as 
\beq \label{Jscaletr}
	J = 1+wN^{-1/3}. 
\eeq
Then for each $w\in \R$, the asymptotic result \eqref{eq:thmspin1low002} still holds. 
Now in the theory of spiked random matrices, the distribution of $\mu_1$ is known to have the transition under the scaling \eqref{Jscaletr}: 
\beq \label{Jscaletr2}
	N^{2/3} \left( \mu_1- 2 \right) \Rightarrow \TW_{1, w}
\eeq
where $\TW_{1,w}$ is a one-parameter family of random variables with the distribution function obtained in Theorems 1.5 and 1.7 of \cite{Bloemendal-Virag1}. See also \cite{Mo} for the Gaussian case and \cite{Feral-Peche07} for a more general class of Wigner matrices.
See Section \ref{sub:transition23}.

For other transitions, by matching the fluctuation scales, we expect that the critical window for the transition between the paramagnetic regime (ii) and the ferromagnetic regime (iii) is $J = \frac{1}{2\beta} + O(N^{-1/2})$ for each $\beta<1$ and that of the transition between the spin glass regime (i) and the paramagnetic regime (ii) is $\beta = \frac{1}{2} + O(\frac{\sqrt{\log N}}{N^{1/3}})$ for each $J<1$. 
However, the analysis of these transition regimes is yet to be done.

\subsection{Organization}

The rest of paper is organized as follows. 
In Section \ref{sec:proofofthmma} and Section \ref{sec:proofmain}, we prove the main results, Theorem \ref{thm:main} and Theorem \ref{thm:main2}, respectively. 
Theorem \ref{thm:linear} (linear statistics) is proved in Section \ref{sec:lspr} assuming Proposition \ref{prop:gaussian xi} and Lemma \ref{lem:Gamolrsm}. 
Proposition \ref{prop:gaussian xi} is proved in Section \ref{sec:outline}--\ref{sec:miscellanies}.
Lemma \ref{lem:Gamolrsm} is proved in Section \ref{sub:nonrandom}.
Certain technical large deviation estimates are proved in Section \ref{subsec:largdeves}.

\begin{nrem}
Throughout the paper we use $C$ or $c$ in order to denote a constant that is independent of $N$. 
Even if the constant is different from one place to another, we may use the same notation $C$ or $c$ as long as it does not depend on $N$ for the convenience of the presentation.
\end{nrem}

\begin{nrem}
The notation $\Rightarrow$ denotes the convergence in distribution as $N\to \infty$. 
\end{nrem}

\begin{nrem}
For random variables $X$ and $Y$ depending on $N$, we use the notation $X \prec Y$ to mean that 
$$
\p(|X| > N^{\epsilon} |Y|) < N^{-D}
$$
for any (small) $\epsilon > 0$ and (large) $D > 0$. The relation $\prec$ is transitive and satisfies the arithmetic rules, e.g., if $X_1 \prec Y_1$ and $X_2 \prec Y_2$ then $X_1 + X_2 \prec Y_1 + Y_2$ and $X_1 X_2 \prec Y_1 Y_2$. We will also use the notation $X = \caO(N^p)$ if $X \prec N^p$ for a constant $p$.
\end{nrem}

\subsubsection*{Acknowledgments}
We would like to thank Zhidong Bai, Zhigang Bao, Wei-Kuo Chen, Jack Silverstein, and Jianfeng Yao for several useful communications. 
The work of Jinho Baik was supported in part by NSF grants DMS1361782. The work of Ji Oon Lee was supported in part by Samsung Science and Technology Foundation project number SSTF-BA1402-04.

\section{Proof of Theorem \ref{thm:main}} \label{sec:proofofthmma}

We already discussed in Section \ref{sec:resultsm} how Theorem \ref{thm:main} (i), (iii) follow from Theorem \ref{thm:main2} and \eqref{spikedr1fl}.
We now check that Theorem \ref{thm:main} (iii) follows from Theorem \ref{thm:main2} and Theorem \ref{thm:linear}. 

In Theorem \ref{thm:linear}, we use the function $\varphi(x)=g(x)= \log\left( 2\beta+ \frac1{2\beta} -x\right)$. 
We first evaluate $M(\varphi)$ and $V(\varphi)$ in Theorem \ref{thm:linear} for this function.

The variance $V(\varphi)$ does not depend on $J$ and $J'$, and hence it is the same as the $J=J'=0$ case. 
The value  $\sigma^2 = \frac14 V(\varphi)$ was evaluated (3.13) of \cite{BaikLee} (see the second last sentence in Section 5 of \cite{BaikLee}); this is equal to $\alpha_1$ in \eqref{eq:alpha1}.

Now consider $M(\varphi)$. 
For the function $\varphi=g$, it was shown in (A.17) of \cite{BaikLee} that 
\beq
	\tau_0(\varphi)=-\log(2\beta), \quad \tau_1(\varphi)=-2\beta, \quad \tau_2(\varphi)=-2\beta^2, \quad \tau_4(\varphi)= -4\beta^4.
\eeq
We now evaluate 
\beq \label{eq:Mvarfospecf}
\begin{split}
	\frac{1}{2\pi \ii} \oint  \varphi \left(-s-\frac{1}{s} \right)  \frac{J^2 s}{1 + Js}  \dd s
	=  \frac{1}{2\pi \ii} \oint  \log \left(2\beta+ \frac1{2\beta}+s+\frac{1}{s} \right)  \frac{J^2 s}{1 + Js}  \dd s.
\end{split} \eeq
Set $B=2\beta$. Then $B<\min\{1, 1/J\}$ since we are in the paramagnetic regime. 
The above integral is 
\beq
	F(B)= \frac{1}{2\pi \ii} \oint_{|s|=r} \log \left(B+\frac1{B}+s+\frac{1}{s} \right)  \frac{J^2 s}{1 + Js}  \dd s
\eeq	
where we can take $r$ to be any number satisfying $B<r<\min\{1, 1/J\}$. 
Its derivative is 
\beq \label{F derivative}
	F'(B)= \frac{1}{2\pi \ii} \oint_{|s|=r} \frac{(B^2-1)J^2 s^2}{B(B+s)(1+Bs)(1 + Js)}  \dd s
	=- \frac{J^2B}{1-JB} = J - \frac{J}{1-JB}
\eeq
by the calculus of residue: the one pole inside the contour is $s=-B$.
Hence $F(B)=JB+ \log(1-JB)+C$ for a constant $C$ for every $B$ satisfying $0<B<\min\{1, 1/J\}$. To find the constant $C$, 
note that 
\beq
	F(B)= \frac{1}{2\pi \ii} \oint_{|s|=r} \log \left( \frac{(B+s)(1+Bs)}{Bs} \right)  \frac{J^2 s}{1 + Js}  \dd s
	= \frac{1}{2\pi \ii} \oint_{|s|=r} \log \left( \frac{(B+s)(1+Bs)}{s} \right)  \frac{J^2 s}{1 + Js}  \dd s
\eeq	
since the integral of $\frac{J^2 s}{1 + Js}$ over the circle $|s|=r$ is zero. Hence $F(B)\to 0$ as $B\to 0$. This implies that $C=0$ and therefore $F(B)= JB+ \log(1-JB)$. This implies that 
\beq
	\frac{1}{2\pi \ii} \oint  \varphi \left(-s-\frac{1}{s} \right)  \frac{J^2 s}{1 + Js}  \dd s
	=2\beta J + \log(1-2\beta J). 
\eeq	
Therefore,
\beq
	M(\varphi)= 
	 \frac{1}{2} \log\left( 1-4\beta^2 \right) -2\beta J'  
	- 2\beta^2 (w_2-2) - 4\beta^4 \left( W_4 - 3 \right) + 2\beta J + \log(1-2\beta J). 
\eeq

We also have (see (A.5) of \cite{BaikLee})
\beq
	\int_{-2}^2 \log\left( 2\beta+\frac1{2\beta}-x \right) \rho(\dd x) = 2\beta^2-\log\left(2\beta \right).
\eeq
Furthermore, applying Theorem \ref{thm:linear} to function $-g''(x)$, we have 
\beq
	- \frac1{N} \sum_i g''(\mu_i) \to \int_{-2}^2 \frac1{(2\beta+\frac1{2\beta}-x)^2} \rho(\dd x)	
	= \frac{4\beta^2}{1-4\beta^2}.
\eeq
in probability (see (A.8) of \cite{BaikLee} for the equality). 
Therefore, Theorem \ref{thm:main2} (ii) implies that 
\beq
	N\left( F_N-\beta^2 \right) \Rightarrow  \mathcal{N} \left(f_1, \alpha_1 \right)
\eeq
where
\beq
	f_1= -\frac12 M(\varphi)+ \log(2\beta) - \frac12 \log \left( \frac{4\beta^2}{1-4\beta^2} \right), 
	\qquad
	\alpha_1= \frac14 V(\varphi). 
\eeq
These are same as \eqref{eq:f1} and \eqref{eq:alpha1}. The proof is complete.

\section{Proof of Theorem \ref{thm:main2}} \label{sec:proofmain}

As we mentioned before, the leading order limit of the free energy \eqref{Flimit} was obtained in \cite{KosterlitzThoulessJones}. 
This is based on the following integral representation for the quenched case, i.e. for fixed matrix $M$. 

\begin{lem}[\cite{KosterlitzThoulessJones}; also Lemma 1.3 of \cite{BaikLee}] \label{lem:inverse laplace}
Let $M$ be an $N \times N$ symmetric matrix with eigenvalues $\mu_1 \geq \mu_2 \geq \dots \geq \mu_N$. 
Then  
\beq \label{integral representation0}
	\int_{S_{N-1}} e^{\beta \langle \bss, M \bss \rangle }\dd \omega_N(\bss) 
	= C_N  \int_{\gamma - \ii \infty}^{\gamma + \ii \infty} e^{\frac{N}{2} G(z)} \dd z, 
	\quad 
	G(z) = 2\beta z - \frac{1}{N} \sum_i \log (z - \mu_i),
\eeq
where $\gamma$ is any constant satisfying $\gamma>\mu_1$, the integration contour is the vertical line from $\gamma-\ii \infty$ to $\gamma+\ii \infty$, the $\log$ function is defined in the principal branch, and 
\beq
	C_N	= \frac{\Gamma(N/2)}{2\pi \ii (N\beta)^{N/2-1}}.
\eeq 
Here $\Gamma(z)$ denotes the Gamma function. 
\end{lem}

Now for the spin system, the eigenvalues $\mu_i$ are random, but using random matrix theory, 
there are precise estimates on these random variables, and we can still apply the method of steepest-descent. 
A formal application of the method of steepest-descent was done in \cite{KosterlitzThoulessJones} and obtained the leading order term. 
In \cite{BaikLee}, we supply necessary estimates and made the result of \cite{KosterlitzThoulessJones} rigorous when $J=0$. 
We furthermore,  extended the analysis to the next order term and obtained limit theorems, Theorem \ref{thm:main} when $J=0$.
It is not explicitly stated in \cite{BaikLee}, but the analysis in it proved Theorem \ref{thm:main2} for $J=0$ as well.
We now follow the similar approach and prove Theorem \ref{thm:main2} for $J>0$.

\subsection{Rigidity estimates of the eigenvalues}

Let $M$ be a Wigner matrix $M$ with non-zero mean in Definition \ref{def:M}. By definition, 
\begin{enumerate}[(a)]
\item For $i \neq j$, $\E M_{ij} = J N^{-1}$, $\E |M_{ij}|^2 = N^{-1} + J^2 N^{-2}$, $\E |A_{ij}|^4 = W_4 N^{-2} + O(N^{-\frac{3}{2}})$. In addition, for Hermitian case, $\E M_{ij}^2 = J^2 N^{-2}$.
\item For $i=j$, $\E M_{ii} = J' N^{-1}$, $\E |M_{ii}|^2 = W_2 N^{-1} + (J'N^{-1})^2$.
\end{enumerate}

For $M$, we have the following precise rigidity estimate for all eigenvalues other than the largest one. 

\begin{lem}[Theorem 2.13 of \cite{EKYY1}, rigidity] \label{lem:rigidity}
For a positive integer $k \in [1, N]$, let $\hat k := \min \{ k, N+1-k \}$. Let $\gamma_k$ be the classical location defined by
\beq\label{eq:classicallocationdef}
\int_{\gamma_k}^{\infty} \dd \rho_{sc} = \frac{1}{N} \left( k - \frac{1}{2} \right).
\eeq
Then,
\beq \label{eq:rigidity}
|\mu_k - \gamma_k| \prec \hat k^{-1/3} N^{-2/3}
\eeq
for all $k=2,3,\dots, N$. 
\end{lem}

The largest eigenvalue $\mu_1$ depends on $J$ and we have the following Dichotomy:

\begin{lem}[Theorem 6.3 of \cite{KY2013_iso}] \label{lem:largest eigenvalue}
\mbox{ }
\begin{enumerate}[(a)]
\item If $J \leq 1$,
\beq \label{sp1sub}
	|\mu_1 - 2| \prec N^{-2/3}
\eeq

\item If $J > 1$, 
\beq
	\left| \mu_1 - (J + \frac{1}{J}) \right| \prec \sqrt{\frac{J-1 + N^{-1/3}}{N}}. 
\eeq
\end{enumerate}
\end{lem}

\subsection{Proof}

We apply the method of steepest-descent to the integral in Lemma \ref{lem:inverse laplace}. 
It is easy to check that $G'(z)$ is an increasing function of $z$ on $(\mu_1, \infty)$, hence 
there exists a unique $\gamma \in (\mu_1, \infty)$ satisfying the equation $G'(\gamma) = 0$: see Lemma 4.1 of \cite{BaikLee}. 
We see in the analysis below that in the spin glass regime and the ferromagnetic regime, $\gamma$ is close to $\mu_1$ with distance of order $O(N^{\epsilon-1})$. 
On the other hand, for the paramagnetic regime, $\gamma$ is away from $\mu_1$ with distance of order $O(1)$. 

\subsubsection{Spin glass regime: $\beta > \frac{1}{2}$ and $J < 1$} \label{sec:spinglass}

\begin{proof}[Proof of Theorem \ref{thm:main2} (i)]
In Theorem 2.11 of \cite{BaikLee}, we obtained a Tracy-Widom limit theorem, Theorem \ref{thm:main2} (i), for general symmetric random matrix $M$ without assuming that the mean is zero. 
This theorem assumes three conditions, Condition 2.3 (Regularity of measure), Condition 2.4 (Rigidity of eigenvalues), and Condition 2.6 (Tracy-Widom limit of the largest eigenvalue).
The proof actually establishes Theorem \ref{thm:main2} (i) under Condition 2.3 and Condition 2.4 first, which then implies Theorem \ref{thm:main} (i) if we add Condition 2.6: See (6.3) in \cite{BaikLee} and then the sentence below it. 
Now for Wigner matrix with non-zero mean $M$, Condition 2.3 and Condition 2.4 are satisfied clearly from Lemma \ref{lem:rigidity} and Lemma \ref{lem:largest eigenvalue}, including the largest eigenvalue. 
Hence Theorem \ref{thm:main2} (i) is proved. 
\end{proof}

\subsubsection{Paramagnetic regime:  $\beta < \frac{1}{2}$ and $\beta < \frac{1}{2J}$}

\begin{proof}[Proof of Theorem \ref{thm:main2} (ii)]
In Theorem 2.10 of \cite{BaikLee}, we proved a Gaussian limit theorem, Theorem \ref{thm:main} (ii), for general symmetric random matrix $M$ without assuming that the mean is zero. 
This theorem assumes three conditions, Condition 2.3 (Regularity of measure), Condition 2.4 (Rigidity of eigenvalues), and Condition 2.5 (Linear statistics of the eigenvalues). 
Similar to the spin glass regime, the proof actually establishes Theorem \ref{thm:main2} (ii) under Condition 2.3 and Condition 2.4 first, which then implies Theorem \ref{thm:main} (ii) if we add Condition 2.5: See (5.27) and (5.29) in \cite{BaikLee}.  
Now for Wigner matrix with non-zero mean $M$, 
Condition 2.4 is not satisfied when $J>1$ due to Lemma \ref{lem:largest eigenvalue}. 
However, we can easily modify the proof of Theorem 2.10 of \cite{BaikLee} for the paramagnetic conditions as we see now. 
The case $J \leq 1$ follows from Theorem 2.10 of \cite{BaikLee} directly, but we consider this case as well here. 

We choose $\gamma$ in Lemma \ref{lem:inverse laplace} as the unique critical value of $G(z)$ on the part of the real line $z \in (\mu_1, \infty)$. In order to evaluate the integral in \eqref{integral representation0}, we introduce a deterministic function
\beq \label{deterministic G}
\widehat G(z) = 2\beta z - \int_{-2}^2 \log(z-x) \dd \rho(x)
\eeq
where $\rho$ is the semicircle measure. Let $\widehat \gamma$ be the critical point of $\widehat G$ in the interval $(2, \infty)$. As in (A.4) of \cite{BaikLee}, it can be easily checked that
\beq \label{deterministic gamma}
	\widehat \gamma = 2\beta + \frac{1}{2\beta}.
\eeq
Recall the definition of $\widehat J$ in \eqref{hat J}. 
Since $\beta < 1/2$ and  $\beta < \frac{1}{2J}$ in the paramagnetic regime, we find that
\beq
	\widehat \gamma > \widehat J,
\eeq
hence $\widehat \gamma > \mu_1$ with high probability.

Recall that $\gamma_1$ is the classical location of the largest eigenvalue as defined in \eqref{eq:classicallocationdef}. Since $|\mu_1 - \gamma_1| = O(1)$ with high probability, Lemma 5.1 and Corollary 5.2 of \cite{BaikLee} hold for this case as well. Then, Corollary 5.3 and Lemma 5.4 of \cite{BaikLee} also hold, which implies the calculations up to (5.27) and (5.29) of \cite{BaikLee}. 
This proves Theorem \ref{thm:main2} (ii). 
\end{proof}

\subsubsection{Ferromagnetic regime: $J > 1$ and $\beta > \frac{1}{2J}$}

In this case, $\widehat \gamma$ in \eqref{deterministic gamma} satisfies $\widehat \gamma < \widehat J$ since $\beta > \frac{1}{2J}$, and hence the proof for the paramagnetic regime does not apply. 
Instead, this case is similar to the spin glass regime and we modify the proof of Theorem 2.11 of \cite{BaikLee}. 
The following lemma shows that $\gamma$ is close to $\mu_1$ up to order $1/N$. This is similar to Lemma 6.1 of \cite{BaikLee}.  

\begin{lem} \label{lem:case2 gamma}
Let $c>0$ be a constant such that $2\beta - \frac{1}{J} > c$ and $J - 1 > c$. Then,
\beq
\frac{1}{3\beta N} \leq \gamma - \mu_1 \leq \frac{2}{cN}.
\eeq
with high probability.
\end{lem}

\begin{proof}
Note that
\beq
G'(z) = 2\beta - \frac{1}{N} \sum_i \frac{1}{z-\mu_i}.
\eeq
Since $G'(z)< 2\beta -  \frac{1}{N(z-\mu_1)}$, 
we find that $G'(\mu_1 + \frac{1}{3\beta N}) < 0$.

Since $G'(z)$ is an increasing function of $z$ on $(\mu_1, \infty)$, it suffices to show that $G'(\mu_1 + \frac{2}{cN}) > 0$. In order to show this, we first notice that
\beq
G'(z) = 2\beta - \frac{1}{N} \frac{1}{z-\mu_1} - \frac{1}{N} \sum_{i=2}^N \frac{1}{z-\mu_i} \geq 2\beta - \frac{1}{N} \frac{1}{z-\mu_1} - \frac{1}{N} \sum_{i=2}^N \frac{1}{\mu_1 -\mu_i}
\eeq
for $z \geq \mu_1$. From Lemma \ref{lem:rigidity}, we may assume that $\mu_k \, (k \geq 2)$ satisfies the rigidity estimate \eqref{eq:rigidity}. Thus, for any $\epsilon > 0$, if $z > \mu_1 > 2$,
\beq \begin{split}
G'(z) &\geq 2\beta - \frac{1}{N} \frac{1}{z-\mu_1} - \frac{1}{N} \sum_{i=2}^N \left( \frac{1}{\mu_1 -\gamma_i} + \hat{i}^{-1/3} N^{-2/3+\epsilon} \right) \\
&\geq 2\beta - \frac{1}{N} \frac{1}{z-\mu_1} - \int_{-2}^2 \frac{\dd \rho(x)}{\mu_1 - x} - C N^{-1+\epsilon} = 2\beta - \frac{1}{N} \frac{1}{z-\mu_1} - \frac{\mu_1 -\sqrt{\mu_1^2 -4}}{2} - C N^{-1+\epsilon}.
\end{split} \eeq
From Lemma \ref{lem:largest eigenvalue}, we thus find that, for any $0 < \delta < \frac{c}{4}$,
\beq \begin{split}
G' \left(\mu_1 + \frac{2}{cN} \right) &\geq 2\beta - \frac{c}{2} - \frac{1}{2} \left( J + \frac{1}{J} - \sqrt{ \left( J + \frac{1}{J} \right)^2 -4} \right) -\delta - C N^{-1+\epsilon} \\
&\geq 2\beta - \frac{1}{J} - c > 0
\end{split} \eeq
with high probability. This proves the lemma.
\end{proof}

The following lemma is a modification of Lemma 6.2 of \cite{BaikLee}. The proof is simpler here due to the fact that $\mu_1$ is away from $\mu_2$ by $O(1)$.

\begin{lem} \label{lem:case2G}
Assume that there exists a constant $c > 0$ such that $2\beta - \frac{1}{J} > c$ and $J - 1 > c$. Let $\gamma$ be the solution of the equation $G'(\gamma) = 0$ in Lemma \ref{lem:case2 gamma}. Then, for any $0 < \epsilon < 1$,
\beq
G(\gamma) = \widehat{G}(\mu_1) + O(N^{-1+\epsilon})
\eeq
with probability. (See \eqref{deterministic G} for the definition of $\widehat{G}$). Moreover, there exist constants $C_0, C_1 > 0$ such that
\beq
C_0 N^{\ell-1} \leq \frac{(-1)^{\ell}}{(\ell-1)!} G^{(\ell)}(\gamma) \leq C_1^{\ell} N^{\ell-1}
\eeq
for all $\ell = 2, 3, \dots$ with probability. Here, $C_0$ and $C_1$ do not depend on $\ell$.
\end{lem}

\begin{proof}
We assume that the eigenvalues $\mu_k \, (k \geq 2)$ satisfies the rigidity estimate \eqref{eq:rigidity}. Then, from Lemma \ref{lem:case2 gamma},
\beq
G(\gamma) = 2\beta \mu_1 - \int_{-2}^2 \log (\mu_1 -x) \dd \rho(x) + O(N^{-1+\epsilon}) = \widehat{G}(\mu_1) + O(N^{-1+\epsilon}) 
\eeq
with probability. Thus, the first part of the lemma holds.

For the second part of the lemma, recall that there exists a constant $\delta > 0$ such that $\gamma - \mu_i > \delta$ for all $i = 2, 3, \dots, N$. Since
\beq
G^{(\ell)}(\gamma) = \frac{(-1)^{\ell} (\ell-1)!}{N (\gamma - \mu_1)^{\ell}} + \frac{(-1)^{\ell} (\ell-1)!}{N} \sum_{i=2}^N \frac{1}{(\gamma - \mu_i)^{\ell}},
\eeq
we can conclude that the second part of the lemma holds.
\end{proof}

\begin{proof}[Proof of Theorem \ref{thm:main2} (iii)]
Using the above two lemmas, the proof of Lemma 6.3 of \cite{BaikLee} applies without any change, and we find that there exists $K \equiv K(N)$ satisfying $N^{-C} < K < C$ for some constant $C > 0$ such that
\beq
\int_{\gamma - \ii \infty}^{\gamma + \ii \infty} e^{\frac{N}{2} G(z)} \dd z = \ii e^{\frac{N}{2} G(\gamma)} K
\eeq
with high probability.
This implies that, as (6.61) of \cite{BaikLee}, 
\beq
Z_N = \frac{\sqrt{N} \beta}{\ii \sqrt{\pi} (2\beta e)^{N/2}} e^{\frac{N}{2} G(\gamma)} K (1 + O(N^{-1}))
\eeq
with high probability. Recall that $\widehat J = J + \frac{1}{J}$. Then, using Lemma \ref{lem:case2G} and evaluating $\widehat{G}$ as in (A.5) of \cite{BaikLee}, we find that
\beq \begin{split} \label{eq:case2FN}
F_N &= \frac{1}{2} [G(\gamma) - 1 - \log (2\beta)] + O(N^{-1} \log N) = \frac{1}{2} [\widehat{G}(\mu_1) - 1 - \log (2\beta)] + O(N^{-1} \log N) \\
&= \frac{1}{2} [\widehat{G}(\widehat J) - 1 - \log (2\beta)] + \frac{1}{2} \widehat{G}' (\widehat J) \cdot (\mu_1 -\widehat J) + O(N^{-1} \log N)\\
&= \beta \left( J + \frac{1}{J} \right) - \frac{1}{4J^2} - \frac{1}{2} \log (2\beta J) - \frac{1}{2} + \left( \beta - \frac{1}{2J} \right) (\mu_1 - \widehat J) + O(N^{-1} \log N),
\end{split} \eeq
hence
\beq \label{eq:case2fluctuation}
F_N - F(\beta) = \left( \beta - \frac{1}{2J} \right) (\mu_1 - \widehat J) + O(N^{-1} \log N),
\eeq
with high probability. This completes the proof. 
\end{proof}

\subsubsection{Transition between spin glass regime and ferromagnetic regime}
\label{sub:transition23}

Consider a fixed $\beta > 1/2$ and $J = 1 + w N^{-1/3}$. As in Section \ref{sec:spinglass}, we can prove Theorem \ref{thm:main2} (i) assuming Condition 2.3 (Regularity of measure) and Condition 2.4 (Rigidity of eigenvalues) of \cite{BaikLee}, and Condition 2.3 and Condition 2.4 are satisfied from Lemma \ref{lem:rigidity} and Lemma \ref{lem:largest eigenvalue}.

As discussed in Section \ref{sec:transitions}, for the Gaussian case
\beq
	N^{2/3} \left( \mu_1- 2 \right) \Rightarrow \TW_{1, w}
\eeq
where $\TW_{1,w}$ is a one-parameter family of random variables with the distribution function obtained in Theorems 1.5 and 1.7 of \cite{Bloemendal-Virag1}. For a non-Gaussian Wigner matrix with non-zero mean, the limit theorem can be proved by applying the Green function comparison method based on the Lindeberg replacement strategy in Theorems 2.4 and 6.3 in \cite{EYY}. The proof of Theorem 6.3 in \cite{EYY} can be reproduced by assuming the rigidity of eigenvalues and the local semicircle law, which hold also for a Wigner matrix with non-zero mean from Lemmas \ref{lem:rigidity}, \ref{lem:largest eigenvalue}, and \ref{lem:local law} (See also Theorem 3.3 and Lemma 3.5 of \cite{LY} for more detail on the case that the variance of the diagonal entries does not match that of GOE.)

\section{Linear statistics} \label{sec:lspr}

\subsection{Proof of Theorem \ref{thm:linear}} \label{sec:prooflinear}

For a function $\varphi$ that satisfies the assumptions of the theorem, we consider $T(\varphi)$, the weak limit of the random variable
\beq
T_N(\varphi) = \sum_{i=1}^N \varphi(\mu_i) - N \int_{-2}^2 \varphi(x) \frac{\sqrt{4-x^2}}{2\pi} \dd x = N \int_{-\infty}^{\infty} \varphi(x) [\rho_N - \rho](\dd x).
\eeq
Fix ($N$-independent) constants $a_- < -2$ and $a_+ > \widehat J$. Let $\Gamma$ be the rectangular contour whose vertices are $(a_- \pm \ii v_0)$ and $(a_+ \pm \ii v_0)$ for some $v_0 \in (0, 1]$. Then,
\beq \label{eq:contour representation}
T_N(\varphi) = \frac{N}{2\pi \ii} \int_{\R} \oint_{\Gamma} \frac{\varphi(z)}{z-x} [\rho_N - \rho](\dd x) \dd z = -\frac{1}{2\pi \ii} \oint_{\Gamma} \varphi(z) \xi_N (z) \dd z
\eeq
where
\beq \label{xidefni}
	\xi_N (z):= N \int_{\R} \frac1{x-z} (\rho_N - \rho)(\dd x).
\eeq
Decompose $\Gamma$ into $\Gamma = \Gamma_u \cup \Gamma_d \cup \Gamma_l \cup \Gamma_r \cup \Gamma_0$, where
\begin{align}
\Gamma_u &= \{ z = x + \ii v_0 : a_- \leq x \leq a_+ \}, \\
\Gamma_d &= \{ z = x - \ii v_0 : a_- \leq x \leq a_+ \}, \\
\Gamma_l &= \{ z = a_- + \ii y : N^{-\delta}\le |y|\leq v_0 \}, \\
\Gamma_r &= \{ z = a_+ + \ii y : N^{-\delta}\le |y|\leq v_0 \}, \\
\Gamma_0 &= \{ z = a_- + \ii y : |y| < N^{-\delta} \} \cup \{ z = a_+ + \ii y : |y| < N^{-\delta} \},
\end{align}
for some sufficiently small $\delta > 0$.

In Sections \ref{sec:outline}--\ref{sec:miscellanies}, we prove the following result for $\xi_N(z)$.

\begin{prop} \label{prop:gaussian xi}
Let 
\begin{equation} \label{sstism}
	s(z)=\int \frac1{x-z} \rho(\dd x)= \frac{-z + \sqrt{z^2 -4}}{2}
\end{equation}
be the Stieltjes transform of the semicircle measure $\rho$. Fix a (small) constant $c > 0$ and a path $\caK \subset \C^+$ such that $\im z > c$ for any $z \in \caK$.
Then, the process $\{ \xi_N(z): z \in \caK \}$ converges weakly to a Gaussian process $\{ \xi(z): z \in \caK \}$ with the mean
\beq \label{eq:mean xi}
b(z) = \frac{s(z)^2}{1-s(z)^2} \left( -J' + \frac{J^2 s(z)}{1 + Js(z)} + (w_2 -1) s(z) + s'(z) s(z) + \left( W_4 - 3 \right) s(z)^3 \right)
\eeq
and the covariance matrix
\beq \label{eq:covariance xi}
\Gamma(z_i, z_j) = s'(z_i) s'(z_j) \left( (w_2 - 2) + 2(W_4 - 3) s(z_i) s(z_j) + \frac{2}{(1-s(z_i) s(z_j))^2} \right).
\eeq
\end{prop}

On the other hand, the following lemma is proved in Section \ref{sub:nonrandom}.  

\begin{lem} \label{lem:Gamolrsm}
For sufficiently small $\delta > 0$,
\beq \label{nonrandom1}
\lim_{v_0 \searrow 0} \limsup_{N \to \infty} \int_{\Gamma_{\sharp}} \E |\xi_N(z) |^2 \dd z = 0,
\eeq
where $\Gamma_{\sharp}$ can be $\Gamma_r$, $\Gamma_l$, or $\Gamma_0$.
\end{lem}

From the explicit formulas \eqref{eq:mean xi} and \eqref{eq:covariance xi}, it is direct to check that 
\beq \label{nonrandom2}
\lim_{v_0 \searrow 0} \int_{\Gamma_{\sharp}} \E |\xi(z)|^2 \dd z = 0.
\eeq

Combining Proposition \ref{eq:covariance xi}, Lemma \ref{lem:Gamolrsm} and \eqref{nonrandom2}, 
we obtain that $T_N(\varphi)$ converges in distribution to a Gaussian random variable $T(\varphi)$ with mean and variance 
\begin{equation} \label{eq:Eintb}
	\E [T(\varphi)]=  -\frac{1}{2\pi \ii} \oint_{\Gamma} \varphi(z) b(z) \dd z, 
	\qquad
	\var [T(\varphi)] = \frac{1}{(2\pi \ii)^2} \oint_{\Gamma} \oint_{\Gamma} \varphi(z_1) \varphi(z_2) 
	\Gamma(z_1, z_2)  \dd z_1 \dd z_2.  
\end{equation}
These integrals are equal to $M(\varphi)$ and $V(\varphi)$ in \eqref{eq:Mvarfo} and \eqref{eq:Vvarfo}: see Lemma \ref{lem:meanvacopm} below. 
This completes the proof of Theorem \ref{thm:linear}.

\begin{rem}
The covariance matrix $\Gamma(z_i, z_j)$ in \eqref{eq:covariance xi} coincides with the one obtained in Proposition 4.1 of \cite{BY2005}. On the other hand, the mean $b_N(z)$ is different from the one in Proposition 3.1 of \cite{BY2005}.
\end{rem}

\subsection{Computation of the mean and variance of $T(\varphi)$} \label{sub:computation}

\begin{lem} \label{lem:meanvacopm}
We have $\E [T(\varphi)]=  M(\varphi)$ and $\var [T(\varphi)] =V(\varphi)$. 
\end{lem}

\begin{proof}
Consider $\var[T(\varphi)]$. Since $\Gamma(z_1, z_2)$ is same as the $J=J'=0$ case, 
$\var[T(\varphi)]$ is same as the one in \cite{BY2005}, and we obtain the result. 
We note that it was further shown in \cite{BY2005} that
\beq \begin{split}
	\covar [ T(\varphi_1), T(\varphi_2) ]
	&= (w_2 -2) \tau_1(\varphi_1) \tau_1(\varphi_2) + (W_4 -3) \tau_2(\varphi_1) \tau_2(\varphi_2) + 2 \sum_{\ell=1}^{\infty} \ell \tau_{\ell}(\varphi_1) \tau_{\ell}(\varphi_2).
\end{split} \eeq

Now let us consider $\E [T(\varphi)]$. 
Recall that (see \eqref{Chebyshev formula}), for $\ell=0,1,2,\dots$, 
\beq
	\tau_\ell(\varphi)
	= \frac1{2\pi} \int_{-\pi}^\pi \varphi(2\cos\theta) \cos(\ell\theta) \, \dd \theta
	= \frac{(-1)^{\ell}}{2\pi \ii} \oint_{|s|=1} \varphi \left(-s-\frac{1}{s} \right) s^{\ell-1} \, \dd s
\eeq
where we set $s=-e^{\ii \theta}$ for the second equality.

We change the variable $z$ to $s=s(z)$ in the first integral in \eqref{eq:Eintb}. 
Note that \eqref{sstism} implies that $s+1/s=-z$ and the map $z\mapsto s$ maps $\C\setminus [-2, 2]$ to the disk $|s|<1$.
Then $\Gamma$ is mapped to a contour with negative orientation that contains $0$ and lies in the slit disk $\Omega:=\{|s|<1\} \setminus [-1, -1/J]$. 
Changing the orientation of the contour, we obtain
\beq \begin{split} \label{Tphi}
	\E[T(\varphi)]&= \frac{1}{2\pi \ii} \oint \varphi \left(-s-\frac{1}{s} \right) \left[ -J' + \frac{J^2 s}{1 + Js} + (w_2-1)s + \frac{s^3}{1-s^2} + \left( W_4 - 3 \right) s^3 \right] \dd s
\end{split} \eeq
along a contour with positive orientation  that contains $0$ and lies in the slit disk $\Omega:=\{|s|<1\} \setminus [-1, -1/J]$.
Note that $\varphi \left(-s-\frac{1}{s} \right)$ is analytic in a neighborhood of the boundary of $\Omega$. 

The first, third, and fifth terms in the integrand of \eqref{Tphi} are, using analyticity, equal to 
\beq \begin{split} \label{mean1}
&\frac{1}{2\pi \ii} \oint_{|s|=1} \varphi \left(-s-\frac{1}{s} \right) \left[ -J' + (w_2-1)s  + \left( W_4 - 3 \right) s^3 \right] \dd s \\
&=J' \tau_1(\varphi) + (w_2-1) \tau_2(\varphi) + \left( W_4 - 3 \right) \tau_4(\varphi).
\end{split} \eeq

For the fourth term in the integrand of \eqref{Tphi}, when we deform the contour to the unit circle, 
then the two poles $s=-1$ and $s=1$ on the circle yields the half of the residue terms and the integral becomes the principal value. The principal value integral is, after setting $-s=e^{\ii\theta}$,  
\beq \begin{split} \label{residue00}
	P.V. \frac{1}{2\pi} \int_{-\pi}^{\pi} \varphi(2\cos \theta) \frac{e^{4\ii \theta}}{1-e^{2\ii \theta}} \dd \theta  
= \frac{1}{2\pi} \int_{-\pi}^{\pi} \varphi(2\cos \theta)\left(-\frac{1}{2} - \cos 2\theta \right) \dd \theta = -\frac{1}{2} \tau_0(\varphi) - \tau_2(\varphi). 
\end{split} \eeq
Hence we obtain 
\beq \begin{split} \label{residue}
	&\frac{1}{2\pi \ii} \oint_{|s|=r} \varphi \left(-s-\frac{1}{s} \right)  \frac{s^3}{1-s^2} \dd s 
	= \frac14 \left( \varphi(-2)+ \varphi(2) \right) -\frac{1}{2} \tau_0(\varphi) - \tau_2(\varphi) .
\end{split} \eeq
From \eqref{Tphi}, \eqref{mean1}, and \eqref{residue}, we proved that $\E [T(\varphi)]=  M(\varphi)$.

\end{proof}

\section{Outline of the proof of Proposition \ref{prop:gaussian xi}} \label{sec:outline}

Sections \ref{sec:outline}--\ref{sec:miscellanies} are dedicated to proving Proposition \ref{prop:gaussian xi}. 
From Theorem 8.1 of \cite{Billingsley_conv}, we need to show (a) the finite-dimensional convergence of $\xi_N(z)$ to a Gaussian vector and (b) the tightness of $\xi_N(z)$. 
To establish the part (a), we compute the limits of the mean $\E [\xi_N(z)]$ and the covariance $\covar [ \xi_N(z_1), \xi_N(z_2)]$ in Section \ref{sec:mean} and \ref{sec:covariance}, respectively. 
The part (a) is concluded in Section \ref{sub:martingale CLT}.
The part (b) is proved in Section \ref{sub:tightness}.

We use the following known results for the resolvent and large deviation estimates. 

\subsection{Local semicircle law and large deviation estimates} \label{sec:prelim}

The Green function (resolvent) of $M$ is $R(z) = (M - zI)^{-1}$.
The normalized trace of the Green function is defined as 
\beq
s_N(z) = \frac{1}{N} \Tr R(z) = \frac{1}{N} \sum_{i=1}^N \frac{1}{\mu_i - z},
\eeq
which is also the Stieltjes transform of $\rho_N$.
Recall that 
\beq
	\xi_N (z) 
	= N \int_{\R} \frac1{x-z} (\rho_N - \rho)(\dd x)
	= N \left( s_N(z) - s(z) \right).
\eeq
We also set 
\beq
	\zeta_N (z) := \xi_N (z) - \E \xi_N (z) = N \left( s_N(z) - \E s_N(z) \right).
\eeq

\begin{lem}[Theorem 2.9 of \cite{EKYY1}, local semicircle law] \label{lem:local law}
Let $\Sigma \geq 3$ be a fixed but arbitrary constant and define the domain $D = \{ z = E + \ii \eta \in \C : |E| \leq \Sigma, \eta \in (0, 3) \}$. Set $\kappa = \min\{ |E-2|, |E+2| \}$. Then, for any $z \in D$ with $\im z = \eta$,
\beq
	|s_N(z) - s(z)| \prec \min \left\{ \frac{1}{N\sqrt{\kappa + \eta}}, \frac{1}{\sqrt N} \right\} + \frac{1}{N\eta}
\eeq
and
\beq \label{eq:Rdelo}
	\max_{i, j} |R_{ij}(z) - \delta_{ij} s(z)| \prec \frac{1}{\sqrt N} + \sqrt{\frac{\im s(z)}{N\eta}} + \frac{1}{N\eta}.
\eeq
\end{lem}

For $\eta \sim 1$, we have the following corollary.

\begin{cor} \label{cor:local law}
Let $\Sigma \geq 3$ be a fixed but arbitrary constant. For a fixed (small) constant $c>0$, define $D_c = \{ z = E + \ii \eta \in \C : |E| \leq \Sigma, \eta \in (c, 3) \}$. Then, for any $z \in D_c$,
\beq \label{sn_s}
	|s_N(z) - s(z)| \prec N^{-1}
\eeq
and
\beq \label{Rbasicestim}
	|R_{ii}(z) - s(z)| \prec N^{-\frac{1}{2}}, \qquad |R_{ij}(z)| \prec N^{-\frac{1}{2}} \quad (i \neq j).
\eeq
Moreover, \eqref{sn_s} holds for any $z \in \Gamma_r \cup \Gamma_l \cup \Gamma_0$ and \eqref{Rbasicestim} holds for any $z \in \Gamma_r \cup \Gamma_l$.
\end{cor}

\begin{proof}
The bounds for $z \in D_c$ are straightforward since $\eta \sim 1$. For $z \in \Gamma_r \cup \Gamma_l \cup \Gamma_0$, from Lemma \ref{lem:rigidity} and Lemma \ref{lem:largest eigenvalue},
\beq \label{eq:SnSforrbasp} \begin{split}
	|s_N(z) - s(z)| &= \left| \frac{1}{N} \sum_{j=1}^N \frac{1}{\mu_j -z} - \int_{\R} \frac{\rho(\dd x)}{x-z} \right| = \left| \frac{1}{N} \sum_{j=1}^N \frac{1}{\gamma_j -z} - \int_{\R} \frac{\rho(\dd x)}{x-z} \right| + \caO(N^{-1}) \\
&= \caO(N^{-1}).
\end{split} \eeq
To prove \eqref{Rbasicestim} for $z \in \Gamma_r$, we notice that
\beq
	\sqrt{\frac{\im s(z)}{N\eta}} \sim \sqrt{\frac{\eta}{\sqrt{\kappa + \eta}} \frac{1}{N\eta}} \sim \sqrt{\frac{1}{N}}
\eeq
since $\kappa \sim 1$ and $N^{-\delta}\le \eta \le v_0$. Since $\frac{1}{N\eta} \le N^{-1+\delta}$, from \eqref{eq:Rdelo}, we find that \eqref{Rbasicestim} holds for $z \in \Gamma_r$. The proof of \eqref{Rbasicestim} for $z \in \Gamma_r$ is the same.
\end{proof}

Let $M^{(a)}$ be the minor of $M$ obtained by removing the $a$-th row and the $a$-th column. We denote by $R^{(a)}$ and $s_N^{(a)}$ the Green function and the averaged Green function of $M^{(a)}$, respectively. It is well known that
\beq \label{schur}
R_{ii} = \frac{1}{M_{ii} - z - \sum_{p, q}^{(i)} M_{ip} R^{(i)}_{pq} M_{qi}}, \qquad R_{ij} = -R_{ii} \sum_p^{(i)} M_{ip} R^{(i)}_{pj} \quad (i \neq j),
\eeq
and
\beq \label{RRdiff}
R_{ij} - R^{(a)}_{ij} = \frac{R_{ia} R_{aj}}{R_{aa}}.
\eeq
Here, $(i)$ in the summation notation means that the index $p = 1, 2, \dots, N$ with $p \neq i$. From the second identity in \eqref{schur}, we also have an estimate
\beq \label{sumofMRest}
\left| \sum_p^{(i)} M_{ip} R^{(i)}_{pj} \right| = \left| \frac{R_{ij}}{R_{ii}} \right| \prec N^{-\frac{1}{2}}
\eeq
for $i \neq j$.

\bigskip

We will also frequently use the following large deviation estimates, which will be proved in Section \ref{subsec:largdeves}.

\begin{lem} \label{lem:M deviation}
Let $S$ be an $(N-1) \times (N-1)$ matrix independent of $M_{ia}$ $(1 \leq a \leq N, a \neq i)$ with matrix norm $\|S\|$. Then, for $n = 1, 2$, there exists a constant $C_n$ depending only on $J$ and $W_4$ in Definition \ref{def:Wigner} and Condition \ref{cond:nonzero} such that
\beq \label{eq:large expectation}
\E \left| \sum_{p, q}^{(i)} M_{ip} S_{pq} M_{qi} - \frac{1}{N} \sum_p^{(i)} S_{pp} \right|^{2n} \leq \frac{C_n \|S\|^{2n}}{N^{n}}.
\eeq
Moreover,
\beq \label{eq:large general}
\left| \sum_{p, q}^{(i)} M_{ip} S_{pq} M_{qi} - \frac{1}{N} \sum_p^{(i)} S_{pp} \right| \prec \frac{\|S\|}{\sqrt N}.
\eeq
\end{lem}

\section{The mean function} \label{sec:mean}

In this section, we assume that $z \in \caK \cup \Gamma_r \cup \Gamma_l$. The estimate for $z \in \Gamma_r \cup \Gamma_l$ will be used later in the proof of Lemma \ref{lem:Gamolrsm}.

Let
\beq
b_N(z) = \E [\xi_N (z) ]= N[ \E s_N(z) - s(z)].
\eeq
From \eqref{schur}, if we set 
\beq \begin{split} \label{Qidefnh}
 	Q_i:= -M_{ii} + \sum_{p, q}^{(i)} M_{ip} R^{(i)}_{pq} M_{qi},
\end{split} \eeq
we have 
\beq \begin{split} \label{schur expand}
 	R_{ii} 
	= \frac{1}{-z-Q_i} 
	&= \frac{1}{-z-s} + \frac{Q_i -s}{(-z-s)^2} +
\frac{\left( Q_i -s \right)^2}{(-z-s)^3} + O \left( \frac{|Q_i-s|^3}{|z+s|^4} \right) \\
	&= s + s^{2}(Q_i -s) + s^3\left( Q_i -s \right)^2 + O \left( |s|^4 |Q_i-s|^3 \right)
\end{split} \eeq
since $s = -1/(s+z)$. 
Using $R_{ii} = \frac{1}{-z-Q_i} $, we have
\beq \begin{split} \label{Qibasesmt}
 	Q_i-s = -\frac1{R_{ii}}-z-s= -\frac1{R_{ii}} +\frac1{s} = \caO(N^{-\frac{1}{2}}).
\end{split} \eeq
We thus find that
\beq \label{b_N}
	b_N = s^2 \sum_i \E (Q_i-s)+ s^3 \sum_i \E (Q_i-s)^2  + O(N^{-\frac{1}{2} + \epsilon}).
\eeq

\subsection{$ \sum_i \E (Q_i-s)$}
We first consider
\beq \label{b1} \begin{split}
	\sum_i \E (Q_i-s) 
	&= \sum_i \E \left[ -M_{ii} + \sum_{p, q}^{(i)} M_{ip} R^{(i)}_{pq} M_{qi} -s \right] \\
	&= -J' + \frac{J^2}{N^2} \E \sum_i \sum_{p, q}^{(i)} R^{(i)}_{pq} + \frac{1}{N} \E \sum_i \sum_p^{(i)} R_{pp}^{(i)} - Ns.
\end{split}
\eeq
Naive power counting shows that the second term is $O(N^{\frac{1}{2}+\epsilon})$ and the third term is $O(N^{\epsilon})$. 
We show that the second term is actually $O(1)$ and the third term is $Ns$ plus an $O(1)$ term.

\subsubsection{$ \frac{1}{N} \E \sum_i \sum_p^{(i)} R_{pp}^{(i)}$}

From \eqref{RRdiff} and \eqref{Rbasicestim}, 
\beq
	\sum_p^{(i)} (R_{pp}^{(i)} - R_{pp}) = - \sum_p^{(i)} \frac{R_{pi} R_{ip}}{R_{ii}} 
	= - \frac{1}{s} \sum_{p}^{(i)} R_{pi} R_{ip} + \caO(N^{-\frac{1}{2}}).
\eeq
This implies
\beq
	\sum_p^{(i)} R_{pp}^{(i)} 
	= \left( \sum_p R_{pp} \right) - R_{ii}   - \frac{1}{s} (R^2)_{pp} + \frac1{s} (R_{ii})^2
	 + \caO(N^{-\frac{1}{2}}),
\eeq
and hence
\beq
	\frac1{N} \sum_i \sum_p^{(i)} R_{pp}^{(i)} 
	= \frac{N-1}{N} \Tr (R) -  \frac{1}{s N} \Tr (R^2) + \frac1{s N } \sum_i (R_{ii})^2
	 + \caO(N^{-\frac{1}{2}}). 
\eeq
Note that by spectral decomposition,
\beq \label{s derivative}
	\frac1{N} \Tr R^2= \frac{1}{N} \sum_i \frac{1}{(\mu_i - z)^2} = \frac{\dd}{\dd z} s_N(z).
\eeq
Since $|s_N(z) - s(z)| \prec N^{-1}$, we find from Cauchy integral formula that $|s_N'(z) - s'(z)| \prec N^{-1+\delta}$. 
Hence, using \eqref{Rbasicestim}, 
\beq \begin{split}
	\frac1{N} \sum_i \sum_p^{(i)} R_{pp}^{(i)} 
	&= (N-1) s_N(z) - \frac{s_N'}{s}   + \frac1{s N } \sum_i (R_{ii})^2
	 + \caO(N^{-\frac{1}{2}}) \\
	 &= N s_N(z) - \frac{s'}{s} 
	 + \caO(N^{-\frac{1}{2}})
\end{split} 
\eeq
We therefore find that 
\beq \label{b13'}
	\frac{1}{N} \E \sum_i \sum_p^{(i)} R_{pp}^{(i)} = Ns + b_N  -\frac{s'}{s} + O(N^{-\frac{1}{2}+\epsilon}).
\eeq

\subsubsection{$\frac{J^2}{N^2} \E \sum_i \sum_{p, q}^{(i)} R^{(i)}_{pq}$}

The case when $p=q$ follows from \eqref{b13'}:
\beq  \label{J2N2Ebla}
	\frac{J^2}{N^2} \E \sum_i \sum_p^{(i)} R_{pp}^{(i)} = J^2 s + O(N^{-1+\epsilon})
\eeq
since a naive estimate shows $b_N=O(N^{\epsilon})$ from the definition.

We now consider the case when $p\neq q$. 
We start with a lemma. 
The strategy of the proof of this lemma is used several places in the paper. 

\begin{lem}\label{lem:unmatch1}
For $q\neq i$,
\beq \label{unmatching offdiagonal0}
	\frac{1}{N} \E \sum_{p}^{(i,q)} R_{pi} R_{iq}^{(p)} = O(N^{-\frac{3}{2} + \epsilon}).
\eeq
\end{lem}

\begin{proof}
For distinct $p,q, i$, we have from \eqref{schur} and \eqref{sumofMRest} that
\beq \begin{split} \label{expRR unmatched}
	&R_{pi} R^{(p)}_{iq} = - R_{pp} \left( \sum_r^{(p)} M_{pr} R^{(p)}_{ri} \right) R^{(p)}_{iq}  
	= -s \left( \sum_r^{(p)} M_{pr} R^{(p)}_{ri} \right) R^{(p)}_{iq} +  \caO(N^{-\frac{3}{2} }) . 
\end{split} \eeq
Hence,
\beq \label{expRRtempo}
\begin{split}
	&\E \left[ R_{pi} R^{(p)}_{iq} \right] 
	= -\frac{Js}{N} \E  \sum_r^{(p)} R^{(p)}_{ri} R^{(p)}_{iq} + O(N^{-\frac{3}{2} + \epsilon}).
\end{split} \eeq
Using $R_{ab}^{(c)}-R_{ab} = \caO(N^{-1})$, which follows from  \eqref{RRdiff} and \eqref{Rbasicestim}, repeatedly, we find that for distinct $p, q, i$, 
\beq \label{expRRtempo0}
\begin{split}
	\sum_r^{(p)} R^{(p)}_{ri} R^{(p)}_{iq} 
	= \sum_r R_{ri} R_{iq}  + \caO(N^{-\frac{1}{2} })
	= \sum_{r}^{(i,q)}R_{ri} R_{iq}^{(r)}  + \caO(N^{-\frac{1}{2} }).
\end{split} \eeq
Summing \eqref{expRRtempo} over $p$, 
this implies that 
\beq \label{expRRtempo00}
\begin{split}
	&\frac1{N} \sum_{p}^{(i,q)} \E \left[ R_{pi} R^{(p)}_{iq} \right] 
	= -Js \frac{N-1}{N^2} \E  \sum_{r}^{(i,q)} R_{ri} R^{(r)}_{iq} + O(N^{-\frac{3}{2} + \epsilon}) .
\end{split} \eeq
Since the two sums on either side are the same, we obtain that
\beq \label{unmatching offdiagonal0'}
	\frac{1 + Js}{N} \E \sum_{p}^{(i,q)} R_{pi} R_{iq}^{(p)} = O(N^{-\frac{3}{2} + \epsilon}).
\eeq

We now claim that $|1+Js| > c'$ uniformly on $\caK \cup \Gamma_r \cup \Gamma_l$ for some ($N$-independent) constant $c' > 0$. Assuming the claim, it is obvious from \eqref{unmatching offdiagonal0'} that the desired lemma holds.

To prove the claim, we first note that, for $J<1$, the claim is trivial since $|1+Js| > 1- J|s| > 1-J$. Thus, we assume that $J>1$.

Let $z = E + \ii \eta$. It is straightforward to check that for $\im z>0$, 
\beq
	\im s(z) \geq C \sqrt{||E| -2| + \eta} \qquad \text{if} \quad |E| < 2 \quad\text{or} \quad ||E| -2| < \eta
\eeq
and
\beq
	\im s(z) \geq \frac{C \eta}{\sqrt{|E| -2| + \eta}} \qquad \text{if} \quad |E| \geq 2 \quad\text{and} \quad ||E| -2| \geq \eta
\eeq
for some $C>0$ independent of $z$. (See, e.g., Lemma 3.4 of \cite{EYY}.) Thus, $|1+Js| \geq |J| \im s \sim 1$, for $z \in \caK$.

Recall that $a_+ > \widehat J \geq 2$. From the definition of $s(z)$, it is direct to see that $s(a_+) > s(\widehat J) = -1/J$. Moreover,
\beq
	\re s(a_+ + \ii \eta) = \re \int_{-2}^2 \frac1{x-a_+-\ii \eta} \rho(\dd x) = \int_{-2}^2 \frac{x-a_+}{(x-a_+)^2 + \eta^2} \rho(\dd x),
\eeq
hence $\re s(a_+ + \ii \eta)$ is an increasing function of $\eta$. Thus, for $z \in \Gamma_u$, 
\beq
	|1+Js| > 1+ J \re s > 1+ Js(a_+) \sim 1.
\eeq

For $z \in \Gamma_l$, it is easy to see that $\re s > 0$, hence $|1+Js| \geq 1 + J \re s > 1$. This completes the proof of the lemma.
\end{proof}

From \eqref{RRdiff} and \eqref{Rbasicestim},
\beq\label{Ripqintermsofdiffrp}
	R^{(i)}_{pq} = R_{pq} - \frac{R_{pi} R_{iq}}{R_{ii}} = R_{pq} - \frac{R_{pi} R_{iq}}{s} + \caO(N^{-\frac{3}{2}}) = R_{pq} - \frac{R_{pi} R^{(p)}_{iq}}{s} + \caO(N^{-\frac{3}{2}}).
\eeq
Hence we conclude from \eqref{Ripqintermsofdiffrp} and \eqref{unmatching offdiagonal0} that
\beq \begin{split} \label{b121}
	\frac{J^2}{N^2} \E \sum_i \sum_{p\neq q}^{(i)} R^{(i)}_{pq} 
	&= \frac{J^2}{N^2} \E \sum_i \sum_{p\neq q}^{(i)} R_{pq} + O(N^{-\frac{1}{2} + \epsilon}) \\
	&= \frac{J^2}{N^2} \E \sum_i \sum_{p\neq q} R_{pq} + O(N^{-\frac{1}{2} + \epsilon}) 
= \frac{J^2}{N} \E \sum_{p\neq q} R_{pq} + O(N^{-\frac{1}{2} + \epsilon}).
\end{split} \eeq
We showed that the upper index $(i)$ after adding a negligible term.

\bigskip

We now compute the right hand side of \eqref{b121}. 
From \eqref{schur} and  \eqref{schur expand}, 
\beq \label{b122} \begin{split}
	R_{pq} 
	&= -R_{pp} \sum_r^{(p)} M_{pr} R^{(p)}_{rq} \\
	&= -s \sum_r^{(p)} M_{pr} R^{(p)}_{rq} - s^2 \left(Q_p -s \right) \sum_r^{(p)} M_{pr} R^{(p)}_{rq} + \caO(N^{-\frac{3}{2} + \epsilon}).
\end{split} \eeq
Taking expectation, the first term becomes, 
\beq \label{b1221}
	-s \E \sum_r^{(p)} M_{pr} R^{(p)}_{rq} 
	= -\frac{Js}{N} \E \sum_r^{(p)} R^{(p)}_{rq} .
\eeq

Since
\beq
	\E M_{pp} \sum_r^{(p)} M_{pr} R^{(p)}_{rq} 
	= \frac{J'J}{N^2}  \E  \sum_r^{(p)}  R^{(p)}_{rq} 
	= O(N^{-\frac{3}{2} + \epsilon}),
\eeq
the second term in \eqref{b122} satisfies, also using \eqref{b1221}, 
\beq \label{b122secondtermap}
	\E \left[ \left(Q_p -s \right) \sum_r^{(p)} M_{pr} R^{(p)}_{rq} \right] 
	= \E \sum_{a, b}^{(p)} M_{pa} R^{(p)}_{ab} M_{bp} \sum_r^{(p)} M_{pr} R^{(p)}_{rq} 
	-\frac{Js}{N} \E \sum_r^{(p)} R^{(p)}_{rq} + O(N^{-\frac{3}{2} + \epsilon}).
\eeq
We now evaluate the term
$$
	\E \sum_{a, b}^{(p)} M_{pa} R^{(p)}_{ab} M_{bp} \sum_r^{(p)} M_{pr} R^{(p)}_{rq},
$$
by considering different choices of the indices $a,b,r$ separately as follows. 
\begin{enumerate}[1)]
\item When $a, b, r$ are all distinct,
\beq \label{Ethreepartfica}
	\E \sum^{(p)} M_{pa} R^{(p)}_{ab} M_{bp} M_{pr} R^{(p)}_{rq}
	= \frac{J^3}{N^3} \sum^{(p)} \E \left[ R^{(p)}_{ab} R^{(p)}_{rq} \right],
\eeq
where the summation is over all distinct $a,b,r$. 
The part of the sum in which the index $r$ is equal to $q$ is  
\beq \label{Ethreepartfitcatemp}
	\frac{J^3}{N^3} \sum^{(p)}_{a\neq b} \E \left[ R^{(p)}_{ab} R^{(p)}_{qq} \right]
	= O(N^{-3/2+\epsilon})
\eeq
by naive estimate. Hence we assume that the index $r$ satisfies $r\neq q$. 
Now similar to Lemma~\ref{lem:unmatch1}, 
for distinct $a,b,r,q, p$,
\beq \begin{split}
	\E \left[ R^{(p)}_{ab} R^{(p)}_{rq} \right] 
	&= \E \left[ R_{ab} R^{(a)}_{rq} \right] + O(N^{-\frac{3}{2} + \epsilon})
	= \E \left[ -R_{aa} \sum_t^{(a)} M_{at} R^{(a)}_{tb} R^{(a)}_{rq} \right] + O(N^{-\frac{3}{2} + \epsilon}) \\
	&= -\frac{Js}{N} \E \sum_t^{(a)} R^{(a)}_{tb} R^{(a)}_{rq} + O(N^{-\frac{3}{2} + \epsilon}) = -\frac{Js}{N} \E \sum_t^{(p)} R^{(p)}_{tb} R^{(p)}_{rq} + O(N^{-\frac{3}{2} + \epsilon}) \\
	&= -\frac{Js}{N} \E \sum_{t: t \neq b, r, q}^{(p)} R^{(p)}_{tb} R^{(p)}_{rq} + O(N^{-\frac{3}{2} + \epsilon}).
\end{split} \eeq
Summing over $a$,
\beq
	\frac{1}{N} \E \sum_{a: a \neq b, r, q}^{(p)} R^{(p)}_{ab} R^{(p)}_{rq} 
	= -Js \frac{N-4}{N^2} \E \sum_{t: t \neq b, r, q}^{(p)} R^{(p)}_{tb} R^{(p)}_{rq} + O(N^{-\frac{3}{2} + \epsilon})
\eeq
for distinct $b,r, q, p$. 
Hence, after adding three $\caO(N^{-1/2})$ terms to the sum,
\beq \label{Rabrq distinct}
	\frac{1}{N} \E \sum_{a}^{(p)} R^{(p)}_{ab} R^{(p)}_{rq} = O(N^{-\frac{3}{2} + \epsilon})
\eeq
for distinct $b, r,q, p$. 
Using this, we find that \eqref{Ethreepartfica} with the summation over all distinct $a,b,r$ with $r\neq q$
is $O(N^{-\frac{3}{2} + \epsilon})$.
Since the case when $r=q$ has the same estimate in \eqref{Ethreepartfitcatemp}, we find that 
\beq\label{caseqofq1of}
	\E \left[ \sum^{(p)} M_{pa} R^{(p)}_{ab} M_{bp} M_{pr} R^{(p)}_{rq} \right] = O(N^{-\frac{3}{2} + \epsilon}),
\eeq
where the summation is over all distinct $a,b,r$.

\item When $a=b \neq r$,
\beq \begin{split} \label{Rabr distinct}
	&\E \sum_{a \neq r}^{(p)} M_{pa} R^{(p)}_{aa} M_{ap} M_{pr} R^{(p)}_{rq}
	= \frac{J}{N} \left(\frac1{N}+\frac{J^2}{N^2}\right) \sum_{a \neq r}^{(p)} \E \left[ R^{(p)}_{aa} R^{(p)}_{rq} \right]\\
	&\qquad = \frac{J}{N^2} \sum_{a, r}^{(p)} \E \left[ R^{(p)}_{aa} R^{(p)}_{rq} \right] + O(N^{-\frac{3}{2} + \epsilon}) 
	= \frac{J}{N} \E \left[ s_N^{(p)}  \sum_r^{(p)} R^{(p)}_{rq} \right]+ O(N^{-\frac{3}{2} + \epsilon})
\end{split} \eeq
where we define
\beq
	s_N^{(p)} = \frac1{N} \sum_a^{(p)} R_{aa}^{(p)}. 
\eeq

\item When $a=r \neq b$ (or $b=r \neq a$),
\beq \label{Rab distinct0}
	\E \sum_{a \neq b}^{(p)} M_{pa} R^{(p)}_{ab} M_{bp} M_{pa} R^{(p)}_{aq}
	= \frac{J}{N} \left(\frac1{N}+\frac{J^2}{N^2}\right)  \sum_{a \neq b}^{(p)} \E \left[ R^{(p)}_{ab} R^{(p)}_{aq} \right] .
\eeq
The part of the sum in which either $a=q$ or $b=q$ is $O(N^{-\frac{3}{2} + \epsilon})$ from naive estimate. 
Now for $a\neq q$, 
\beq \label{Rab distinct1}
	\frac{1}{N} \sum_{b}^{(p,a,q)} R^{(p)}_{ab} R^{(p)}_{aq} 
	= \frac{1}{N} \sum_{b}^{(p,a,q)} R_{ab} R^{(b)}_{aq} + \caO(N^{-3/2})
	= \frac{1}{N} \sum_{b}^{(a,q)} R_{ab} R^{(b)}_{aq} + \caO(N^{-3/2})
\eeq
Following the proof of \eqref{unmatching offdiagonal0}, we can check that $\E \left[ \frac{1}{N} \sum_{b}^{(a,q)} R^{(p)}_{ba} R^{(b)}_{aq} \right] = O(N^{-\frac{3}{2} + \epsilon})$. (This is easy to see for a real symmetric matrix since $R_{ab}=R_{ba}$.) Thus,
\beq \label{Rab distinct2}
	\E \sum_{a \neq b}^{(p)} M_{pa} R^{(p)}_{ab} M_{bp} M_{pa} R^{(p)}_{aq}
	= \frac{J}{N^2} \sum_{a}^{(p,q)} \E \left[ \sum_{b}^{(a,q)} R^{(p)}_{ba} R^{(b)}_{aq} \right] + O(N^{-\frac{3}{2} + \epsilon})
	= O(N^{-\frac{3}{2} + \epsilon}).
\eeq

\item When $a=b=r$,
\beq \label{Rabr same}
	\E \sum_r^{(p)} M_{pr} R^{(p)}_{rr} M_{rp} M_{pr} R^{(p)}_{rq}
	= \left( \frac{W_3}{N^{\frac{3}{2}}}+ \frac{J^3}{N^3} \right)  \E \sum_r^{(p)} R^{(p)}_{rr} R^{(p)}_{rq}
	= \frac{W_3 s}{N^{\frac{3}{2}}} \E \sum_r^{(p)} R^{(p)}_{rq} + O(N^{-\frac{3}{2} + \epsilon}).
\eeq
\end{enumerate}

Putting the above four cases into \eqref{b122secondtermap}, we find that
\beq \begin{split}
	\E \left[ \left( Q_p -s \right) \sum_r^{(p)} M_{pr} R^{(p)}_{rq} \right]
	&  = \frac{J}{N} \E \left[ \left( s_N^{(p)} - s \right) \sum_r^{(p)} R^{(p)}_{rq} \right] + \frac{W_3 s}{N^{\frac{3}{2}}} \E \sum_r^{(p)} R^{(p)}_{rq} + O(N^{-\frac{3}{2} + \epsilon}).
\end{split} \eeq
Note that 
\beq
	s_N^{(p)} - s = \frac1{N} \sum_{a}^{(p)} \left( R_{aa}^{(p)}-R_{aa}\right) - \frac1{N} R_{pp} +  (s_N -s)
	= \caO(N^{-1}).
\eeq
Hence,
\beq \label{b1222} \begin{split}
	\E \left[ \left( Q_p -s \right) \sum_r^{(p)} M_{pr} R^{(p)}_{rq} \right] 
	=  \frac{W_3 s}{N^{\frac{3}{2}}} \E \sum_r^{(p)} R^{(p)}_{rq} + O(N^{-\frac{3}{2} + \epsilon}) .
\end{split} \eeq

\bigskip

From \eqref{b122}, \eqref{b1221}, and \eqref{b1222}, for $p\neq q$, 
\beq \begin{split}
	\E  \left[ R_{pq}  \right]
	&= - \left( \frac{J s}{N} + \frac{W_3 s^3}{N^{\frac{3}{2}}} \right) \E \sum_r^{(p)} R^{(p)}_{rq}+ O(N^{-\frac{3}{2} +\epsilon}).
\end{split} \eeq
Using \eqref{Ripqintermsofdiffrp} and \eqref{unmatching offdiagonal0}, this implies that 
\beq \begin{split}
	\E  \left[ R_{pq}  \right]
	&= - \left( \frac{J s}{N} + \frac{W_3 s^3}{N^{\frac{3}{2}}} \right) \E \sum_r R_{rq}+ O(N^{-\frac{3}{2} + \epsilon}).
\end{split} \eeq
From this we find that 
\beq \begin{split}
	\frac{1}{N} \E \sum_p R_{pq} 
	&= \frac{1}{N} \E \left[ \sum_p^{(q)} R_{pq} + R_{qq} \right] 
	= - \left( \frac{J s}{N} +\frac{ W_3 s^3}{N^{\frac{3}{2}}} \right) \E \sum_r R_{rq} + \frac{s}{N} + O(N^{-\frac{3}{2} + \epsilon}),
\end{split} \eeq
which implies that 
\beq \label{R unmatched}
	\frac{1}{N} \E \sum_p R_{pq} = \frac{s}{(1 + Js) N} + O(N^{-\frac{3}{2} + \epsilon}).
\eeq
Therefore, we obtain
\beq \label{b12} \begin{split}
	\frac{J^2}{N} \E \sum_{p \neq q} R_{pq} 
	&= \frac{J^2}{N} \E \sum_{p, q} R_{pq} - \frac{J^2}{N} \E \sum_{p} R_{pp} 	 
	= \frac{J^2 s}{1 + Js} - J^2 s + O(N^{-\frac{1}{2}+ \epsilon}).
\end{split}\eeq

We obtain from \eqref{J2N2Ebla}, \eqref{b121}, and \eqref{b12} that
\beq \label{b1'222}
	\frac{J^2}{N^2} \E \sum_i \sum_{p, q}^{(i)} R^{(i)}_{pq}
	= \frac{J^2 s}{1 + Js}  + O(N^{-\frac{1}{2} + \epsilon}).
\eeq

\subsubsection{Conclusion for $\sum_i \E (Q_i-s)$}

From \eqref{b1}, \eqref{b13'}, and \eqref{b1'222}, 
\beq \label{b1'}
	\sum_i \E (Q_i-s)  
	= -J'  + b_N - \frac{s'}{s} + \frac{J^2 s}{1 + Js} + O(N^{-\frac{1}{2} + \epsilon}).
\eeq

\subsection{$\sum_i \E (Q_i-s)^2$}

We next turn to the second term in \eqref{b_N}. We begin with
\beq \begin{split} \label{b2}
	\E (Q_i-s)^2
	= & \frac{w_2}{N} + \frac{(J')^2}{N^2}+ \frac{2J' s}{N} + s^2 
	- 2 \left( s+ \frac{J'}{N} \right) \E \sum_{p, q}^{(i)} M_{ip} R^{(i)}_{pq} M_{qi}   \\
	& + \E \sum_{p, q, r, t}^{(i)} M_{ip} R^{(i)}_{pq} M_{qi} M_{ir} R^{(i)}_{rt} M_{ti}.
\end{split} \eeq
The first sum on the right hand side satisfies 
\beq \label{b22}
\E \sum_{p, q}^{(i)} M_{ip} R^{(i)}_{pq} M_{qi} = \frac{1}{N} \E \sum_p^{(i)} R^{(i)}_{pp} + \frac{J^2}{N^2} \E \sum_{p, q}^{(i)} R^{(i)}_{pq} = \E s_N^{(i)} + \frac{J^2 s}{(1 + Js)N} + O(N^{-\frac{3}{2} + \epsilon}),
\eeq
using \eqref{R unmatched} (applied to the Green function of an $(N-1) \times (N-1)$ matrix).

\bigskip

\subsubsection{Computation of $\E\big[ \sum_{p, q, r, t}^{(i)} M_{ip} R^{(i)}_{pq} M_{qi} M_{ir} R^{(i)}_{rt} M_{ti} \big]$}

In order to evaluate the last term in \eqref{b2},
we consider several cases separately.

\begin{enumerate}[1)]
\item When $p, q, r, t$ are all distinct,
$$
	\sum^{(i)} \E \left[ M_{ip} R^{(i)}_{pq} M_{qi} M_{ir} R^{(i)}_{rt} M_{ti} \right] 
	=  \frac{J^4}{N^4} \sum^{(i)} \E \left[ R^{(i)}_{pq} R^{(i)}_{rt} \right] 
	= O(N^{-\frac{3}{2} + \epsilon})
$$
due to \eqref{Rabrq distinct}. 
Here the sum is taken over all distinct $p,q,r,t$.

\item When $| \{p, q, r, t \}| = 3$:
\begin{enumerate}
\item If $p=q$,
$$
	\E \left[ M_{ip} R^{(i)}_{pp} M_{pi} M_{ir} R^{(i)}_{rt} M_{ti} \right] 
	= \frac{J^2}{N^2}\left( \frac1{N}+ \frac{J^2}{N^2}\right) \E \left[ R^{(i)}_{pp} R^{(i)}_{rt} \right] 
	= \frac{J^2}{N^3} \E \left[ R_{pp} R_{rt}^{(i)} \right] + O(N^{-\frac{9}{2} + \epsilon}).
$$
Thus, using \eqref{Rbasicestim} and \eqref{b12}, we find that 
\beq \begin{split}
	& \sum^{(i)} \E \left[ M_{ip} R^{(i)}_{pp} M_{pi} M_{ir} R^{(i)}_{rt} M_{ti} \right] 
	= \frac{J^2}{N^2}  \sum^{(i)}_{r \neq t}   \E \left[  s_N R_{rt}^{(i)} \right]  + O(N^{-\frac{3}{2} + \epsilon}) \\
	&\qquad =\frac{J^2 s}{N^2}  \sum^{(i)}_{r \neq t}   \E \left[  R_{rt}^{(i)} \right]  + O(N^{-\frac{3}{2} + \epsilon})
	= -\frac{J^3 s^3}{(1 + Js)N} + O(N^{-\frac{3}{2} + \epsilon}).
\end{split} \eeq
where the first sum is over all distinct $p,r,t$.

\item If $r=t$, the calculation is the same as the above.

\item Other cases have negligible contributions, i.e., bounded by $N^{-\frac{3}{2} + \epsilon}$, due to unmatching off-diagonal terms using \eqref{unmatching offdiagonal0} and the derivation is similar to that of \eqref{Rab distinct2}.

\end{enumerate}

\item When $| \{p, q, r, t \}| = 2$:

\begin{enumerate}

\item If there is a triplet, e.g., $p=q=r$, the contribution is $O(N^{-\frac{3}{2} + \epsilon})$. 
For example, 
\beq \begin{split}
	\E \sum_{p \neq t}^{(i)} M_{ip} R^{(i)}_{pp} M_{pi} M_{ip} R^{(i)}_{pt} M_{ti} 
	&= \frac{J}{N} \left( \frac{W_3}{N^{\frac{3}{2}}}+ \frac{J^3}{N^3} \right)  
	\E \sum_{p \neq t}^{(i)} R^{(i)}_{pp} R^{(i)}_{pt} \\
	&= \frac{W_3 Js}{N^{\frac{5}{2}}} \E \sum_{p \neq t}^{(i)} R^{(i)}_{pt} + O(N^{-\frac{3}{2} + \epsilon}) 
	= O(N^{-\frac{3}{2} + \epsilon}),
\end{split} \eeq
where we used \eqref{b12}.

\item If $p=q$ and $r=t$,  
\beq \begin{split}
	&\E \sum_{p \neq r}^{(i)} M_{ip} R^{(i)}_{pp} M_{pi} M_{ir} R^{(i)}_{rr} M_{ri} 
	= \left( \frac{1}{N} + \frac{J^2}{N^2} \right)^2 \E \sum_{p \neq r}^{(i)} R^{(i)}_{pp} R^{(i)}_{rr}\\
	&= \left( \frac{1}{N} + \frac{J^2}{N^2} \right)^2  \E \sum_p^{(i)} R_{pp}^{(i)} \left( Ns_N^{(i)}-R_{pp}^{(i)}\right)  
	= \E \left( s_N^{(i)} \right)^2 - \frac{s^2}{N} + \frac{2 J^2 s^2}{N} + O(N^{-\frac{3}{2} + \epsilon}).
\end{split} \eeq

\item If $p=t$ and $q=r$,
\beq \begin{split}
	&\E \sum_{p \neq q}^{(i)} M_{ip} R^{(i)}_{pq} M_{qi} M_{iq} R^{(i)}_{qp} M_{pi} 
	= \left( \frac{1}{N} + \frac{J^2}{N^2} \right)^2  \E \sum_{p \neq q}^{(i)} R^{(i)}_{pq} R^{(i)}_{qp} \\ 
	&= \left( \frac{1}{N} + \frac{J^2}{N^2} \right)^2 \left[  \E \Tr (R^{(i)})^2 - \E \sum_p^{(i)} ( R^{(i)}_{pp})^2 \right]
	= \frac{s'}{N} - \frac{s^2}{N} + O(N^{-\frac{3}{2} + \epsilon}),
\end{split} \eeq
where we used \eqref{s derivative} (applied to an $(N-1)\times (N-1)$ matrix).

\item If $p=r$ and $q=t$, the expectation $\E [M_{ip} R^{(i)}_{pq} M_{qi} M_{ip} R^{(i)}_{pq} M_{qi}]$ is negligible when $M$ is complex Hermitian. When $M$ is real symmetric, the calculation is the same as the above, since $R$ is also symmetric and the contribution is
\beq \begin{split}
\frac{s'}{N} - \frac{s^2}{N} + O(N^{-\frac{3}{2} + \epsilon}).
\end{split} \eeq

\end{enumerate}

\item When $p=q=r=t$,
\beq
\E \sum_p^{(i)} M_{ip} R^{(i)}_{pp} M_{pi} M_{ip} R^{(i)}_{pp} M_{pi} = \frac{W_4 s^2}{N} + O(N^{-\frac{3}{2} + \epsilon}).
\eeq

\end{enumerate}

Combining all cases together, we obtain
\beq \begin{split} \label{b26}
	&\E \sum_{p, q, r, t}^{(i)} M_{ip} R^{(i)}_{pq} M_{qi} M_{ir} R^{(i)}_{rt} M_{ti} \\
&= -\frac{2 J^3 s^3}{(1 + Js)N} + \E \left( s_N^{(i)} \right)^2 - \frac{s^2}{N} + \frac{2 J^2 s^2}{N} + \frac{2s'}{N} - \frac{2s^2}{N} + \frac{W_4 s^2}{N} + O(N^{-\frac{3}{2} + \epsilon})
\end{split} \eeq
(when $M$ is real symmetric. For complex Hermitian $M$, we have $\frac{s'}{N} - \frac{s^2}{N}$ instead of $\frac{2s'}{N} - \frac{2s^2}{N}$.)

\subsubsection{Conclusion for  $\sum_i \E \left( Q_i-s\right)^2$}

From \eqref{b2}, \eqref{b22}, and \eqref{b26},
\beq \begin{split} \label{b2'01}
	\E \left( Q_i-s\right)^2
	=& s^2+ \E \left( s_N^{(i)} \right)^2 - 2\left( s+ \frac{J'}{N} \right) \E s_N^{(i)} 
+ \frac{2J's}{N} \\
	&+ \frac{1}{N} \left( w_2 - 3 s^2 + 2s' + W_4 s^2 \right) + O(N^{-\frac{3}{2} + \epsilon}).
\end{split} \eeq
Using $|s_N^{(i)} - s| \prec N^{-1}$ and summing over $i$, we obtain 
\beq \begin{split} \label{b2'}
	\sum_i \E \left( Q_i-s\right)^2 
&= w_2 - 3s^2 + 2s' + W_4 s^2  + O(N^{-\frac{1}{2} + \epsilon}).
\end{split} \eeq

\subsection{Formula of $b_N$}

Inserting \eqref{b1'} and \eqref{b2'} into \eqref{b_N}, we obtain
\beq
b_N = -s^2 J' - s' s + b_N s^2 + \frac{J^2 s^3}{1 + Js} + w_2 s^3 + 2s' s^3 + \left( W_4 - 3 \right) s^5 + O(N^{-\frac{1}{2} + \epsilon}).
\eeq
Therefore, 
\beq
b_N = \frac{s^2}{1-s^2} \left( -J' - \frac{s'}{s} + \frac{J^2 s}{1 + Js} + w_2 s + 2s' s + \left( W_4 - 3 \right) s^3 \right) + O(N^{-\frac{1}{2} + \epsilon}).
\eeq
Using the algebraic identity $s' = \frac{s^2}{1-s^2}$, we can express
\beq
b_N = \frac{s^2}{1-s^2} \left( -J' + \frac{J^2 s}{1 + Js} + (w_2-1)s + s's + \left( W_4 - 3 \right) s^3 \right) + O(N^{-\frac{1}{2} + \epsilon}).
\eeq
This converges to $b(z)$ in Proposition \ref{prop:gaussian xi}. We remark that, when $J' = J = 0$, this reduces to
\beq
b = (1+s')s^3 \left( (w_2-1) + s' + \left( W_4 - 3 \right) s^2 \right),
\eeq
which is the same as Proposition 3.1 of \cite{BY2005}.

\section{The covariance function} \label{sec:covariance}

\subsection{Martingale decomposition}

Following \cite{BY2005}, we consider the filtration
\beq
\caF_k = \sigma( M_{ij}, \, k < i, j \leq N )
\eeq
for $k = 0, 1, \dots, N$ and the conditional expectation
\beq
\E_k( \cdot ) = \E ( \cdot | \caF_k ).
\eeq
Recall that
\beq
\zeta_N = \xi_N - \E \xi_N = \Tr R - \E \Tr R.
\eeq
We use the following martingale decomposition:
\beq \begin{split} \label{eq:zeta_N decomposition0}
	\zeta_N &= \sum_{k=1}^N (\E_{k-1} \Tr R - \E_k \Tr R) 
	= \sum_{k=1}^N (\E_{k-1} - \E_k) \Tr R = \sum_{k=1}^N (\E_{k-1} - \E_k) (\Tr R - \Tr R^{(k)}).
\end{split} \eeq
From \eqref{RRdiff} and \eqref{schur}, 
\beq \begin{split} \label{eq:zeta_N decomposition1}
	\Tr R - \Tr R^{(k)}
	= R_{kk} + \sum^{(k)}_i \frac{R_{ik} R_{ki}}{R_{kk}} 
	= R_{kk} + \sum_i^{(k)} R_{kk} \sum_{p, q}^{(k)} M_{kp} R^{(k)}_{pi} R^{(k)}_{iq} M_{qk}.
\end{split} \eeq
Hence 
\beq \begin{split} \label{eq:zeta_N decomposition}
	\zeta_N
	&= \sum_{k=1}^N (\E_{k-1} - \E_k) \left[ R_{kk} \left( 1 + \sum_{p, q}^{(k)} M_{kp} (R^{(k)})^2_{pq} M_{qk} \right) \right].
\end{split} \eeq

As in the previous section, we expand $R_{kk}$ using Schur formula. Since
\beq
\left| \sum_{p, q}^{(k)} M_{kp} (R^{(k)})^2_{pq} M_{qk} \right| \prec 1
\eeq
from \eqref{eq:large general}, 
it is tempting to speculate that one needs to expand $R_{kk}$ up to third order term, i.e., up to the term of order $N^{-1}$. However, for any random variables $X_k$ and $X_{\ell}$ with $k > \ell$ adapted to the filtration, 
\beq \begin{split}
&\E \left[ (\E_{k-1} - \E_k)X_k \cdot (\E_{\ell-1} - \E_{\ell}) \ol X_{\ell} \right] = \E \left[ \E_{k-1}[ (\E_{k-1} - \E_k)X_k \cdot (\E_{\ell-1} - \E_{\ell}) \ol X_{\ell}] \right] \\
&= \E \left[ (\E_{k-1} - \E_k)X_k \cdot \E_{k-1}[(\E_{\ell-1} - \E_{\ell}) \ol X_{\ell}] \right] = 0.
\end{split} \eeq
Thus,
\beq \label{eq:E_k sum change}
\E \left| \sum_{k=1}^N (\E_{k-1} - \E_k)X_k \right|^2 = \E \sum_{k=1}^N \left| (\E_{k-1} - \E_k)X_k \right|^2.
\eeq
This implies, in particular, that if a random variable $Y_k = \caO(N^{-1})$, then
\beq  \label{eq:E_k sum change2}
\sum_{k=1}^N (\E_{k-1} - \E_k) (X_k + Y_k) = \sum_{k=1}^N (\E_{k-1} - \E_k) X_k + \caO_p (N^{-\frac{1}{2}}),
\eeq
where $\caO_p (N^{-\frac{1}{2}})$ means that the other terms are bounded by $N^{-\frac{1}{2} + \epsilon}$ in probability. 
Applying the argument to the expansion \eqref{schur expand} of $R_{kk}$ in \eqref{eq:zeta_N decomposition}, we find that
\beq \begin{split} \label{zeta}
	\zeta_N &= s \sum_{k=1}^N (\E_{k-1} - \E_k) \left[ \left( 1 + \sum_{p, q}^{(k)} M_{kp} (R^{(k)})^2_{pq} M_{qk} \right) \right] \\
	&\qquad + s^2 \sum_{k=1}^N (\E_{k-1} - \E_k) \left[ \left(Q_k -s \right) \left( 1 + \sum_{p, q}^{(k)} M_{kp} (R^{(k)})^2_{pq} M_{qk} \right) \right] + \caO_p (N^{-\frac{1}{2}}).
\end{split} \eeq
where $Q_k= -M_{kk} + \sum_{r, t}^{(k)} M_{kr} R^{(k)}_{rt} M_{tk}$ as in \eqref{Qidefnh}. 

\subsubsection{First term}

The first term on the right hand side of \eqref{zeta} is given by
\beq \begin{split} \label{z1}
& (\E_{k-1} - \E_k) \sum_{p, q}^{(k)} M_{kp} (R^{(k)})^2_{pq} M_{qk} \\
&= \E_{k-1} \left[ \sum_{p, q}^{(k)} M_{kp} (R^{(k)})^2_{pq} M_{qk}  \right] -  \E_{k-1} \left[ \frac{J^2}{N^2} \sum_{p, q}^{(k)} (R^{(k)})^2_{pq} + \frac{1}{N} \sum_p^{(k)} (R^{(k)})^2_{pp} \right] \\
&=  \E_{k-1} \left[ \sum_{p, q}^{(k)} M_{kp} (R^{(k)})^2_{pq} M_{qk} - \frac{J^2}{N^2} \sum_{p, q}^{(k)} (R^{(k)})^2_{pq} -s' \right] + \caO(N^{-1}).
\end{split} \eeq
This corresponds to $b_k$ of \cite{BY2005}. 

\subsubsection{Second term}

In order to compute the second term in the right hand side of \eqref{zeta}, 
note that 
\beq
	|Q_k-s|  \left| \left( 1 + \sum_{p, q}^{(k)} M_{kp} (R^{(k)})^2_{pq} M_{qk} \right) - \left( 1 + \frac{1}{N} \sum_p^{(k)} (R^{(k)})^2_{pp} \right) \right| \prec \frac{1}{N}.
\eeq
from \eqref{Qibasesmt} and \eqref{eq:large general}, since $\| R^{(k)} \| \leq \frac{1}{\im z}$ and $z \in \caK$.
Thus, the summand in the second term is given by
\beq \begin{split} \label{EkEkqQ1k}
	&  (\E_{k-1} -\E_k ) \left[ (Q_k-s)  \left( 1 + \frac{1}{N} \sum_p^{(k)} (R^{(k)})^2_{pp} \right)\right] +  \caO(N^{-1}).
\end{split} \eeq
Now
\beq \begin{split} 
	& \E_k \left[ (Q_k-s)  \left( 1 + \frac{1}{N} \sum_p^{(k)} (R^{(k)})^2_{pp} \right)\right] \\
	&= \E_{k-1} \left[ \left( - \frac{J'}{N} + \frac{J^2}{N^2} \sum_{r, t}^{(k)} R^{(k)}_{rt} + \frac{1}{N} \sum_{r}^{(k)} R^{(k)}_{rr}-s \right) \left( 1 + \frac{1}{N} \sum_p^{(k)} (R^{(k)})^2_{pp} \right)\right] \\
	&= \E_{k-1} \left[ \left(  \frac{J^2}{N^2} \sum_{r, t}^{(k)} R^{(k)}_{rt}  \right) \left( 1 + \frac{1}{N} \sum_p^{(k)} (R^{(k)})^2_{pp} \right)\right] + \caO(N^{-1}).
\end{split} \eeq
Hence, \eqref{EkEkqQ1k} becomes
\beq \begin{split} \label{z2}
	\E_{k-1} \left( -M_{kk} + \sum_{r, t}^{(k)} M_{kr} R^{(k)}_{rt} M_{tk} - \frac{J^2}{N^2} \sum_{r, t}^{(k)} R^{(k)}_{rt} - s \right) (1 + s') + \caO(N^{-1}).
\end{split} \eeq

\subsubsection{Simplified formula of the martingale decomposition}

From \eqref{zeta}, \eqref{z1}, and \eqref{z2}, we find that 
\beq \label{phi_decompose}
	\zeta_N = \sum_{k=1}^N \E_{k-1} \phi_k + \caO_p(N^{-\frac{1}{2}}),
\eeq
where
\beq \begin{split} \label{def_phi}
	\phi_k &:= s \left( \sum_{p, q}^{(k)} M_{kp} (R^{(k)})^2_{pq} M_{qk} - \frac{J^2}{N^2} \sum_{p, q}^{(k)} (R^{(k)})^2_{pq} - s' \right) \\
&\qquad + s^2 (1+s') \left( -M_{kk} + \sum_{p, q}^{(k)} M_{kp} R^{(k)}_{pq} M_{qk} - \frac{J^2}{N^2} \sum_{p, q}^{(k)} R^{(k)}_{pq} - s \right).
\end{split} \eeq
Since $\frac{\dd}{\dd z} R^{(k)}= (R^{(k)})^2$ and $s'=s^2(1+s')$, this can also be written as 
\beq \label{phi_k}
	\phi_k = \frac{\partial}{\partial z} \left[ s \left( -M_{kk} + \sum_{p, q}^{(k)} M_{kp} R^{(k)}_{pq} M_{qk} - \frac{J^2}{N^2} \sum_{p, q}^{(k)} R^{(k)}_{pq} - s \right) \right].
\eeq
Note that $\phi_k \prec N^{-\frac{1}{2}}$.

\subsection{Covariance}

Let $z_1, z_2, \dots, z_p$ are $p$ distinct points in $\caK$. In order to prove the finite dimensional convergence of $\xi_N$, it suffices to show that the random vector $(\zeta_N (z_1), \zeta_N (z_2), \dots, \zeta_N (z_p))$ converges weakly to a $p$-dimensional mean-zero Gaussian distribution with the covariance matrix $\Gamma(z_i, z_j)$ defined in \eqref{eq:covariance xi}. To prove it, we use the martingale CLT for $\sum_k \E_{k-1} \phi_k$.

Let $z_1$ and $z_2$ be two distinct points in $\caK$.
Following \cite{BY2005}, we consider
\beq
	\Gamma_N (z_1, z_2) = \sum_{k=1}^N \E_k \left[ \E_{k-1} [\phi_k (z_1)] \cdot \E_{k-1} [\phi_k (z_2)] \right].
\eeq
For simplicity, we introduce the notations
\beq
	s_1 = s(z_1), \qquad s_2 = s(z_2).
\eeq
Let
\beq \begin{split} \label{tilde Gamma}
	\wt \Gamma_N (z_1, z_2) &= \sum_{k=1}^N \E_k \left[ \E_{k-1} \left[ -M_{kk} + \sum_{p, q}^{(k)} M_{kp} R^{(k)}_{pq}(z_1) M_{qk} - \frac{J^2}{N^2} \sum_{p, q}^{(k)} R^{(k)}_{pq}(z_1) - s_1 \right] \right.\\
&\left. \qquad\qquad\qquad \times \E_{k-1} \left[ -M_{kk} + \sum_{p, q}^{(k)} M_{kp} R^{(k)}_{pq}(z_2) M_{qk} - \frac{J^2}{N^2} \sum_{p, q}^{(k)} R^{(k)}_{pq}(z_2) - s_2 \right] \right]
\end{split} \eeq
so that 
\beq
	\Gamma_N (z_1, z_2) = \frac{\partial^2}{\partial z_1 \partial z_2} \left[ s_1 s_2 \wt \Gamma_N (z_1, z_2) \right].
\eeq

\subsubsection{The easy parts of $\wt \Gamma_N (z_1, z_2)$}

We now find the limit of $\wt \Gamma_N (z_1, z_2)$. 
In order to simplify notations, let us write
\beq
	\Aq_k(z):= \sum_{p, q}^{(k)} M_{kp} R^{(k)}_{pq}(z) M_{qk},
	\qquad
	\Aw_k(z):= \frac{J^2}{N^2} \sum_{p, q}^{(k)} R^{(k)}_{pq}(z).
\eeq
Then a summand in the formula of $\wt \Gamma_N (z_1, z_2)$ is 
\beq \begin{split} \label{wrGamsumd}
	\E_k \left[ \E_{k-1} \left[ -M_{kk} +\Aq_k(z_1)-\Aw_k(z_1) - s_1 \right] \cdot \E_{k-1} \left[ -M_{kk} + \Aq_k(z_2)-\Aw_k(z_2) - s_2 \right] \right] .
\end{split} \eeq

We first estimate $\Aq_k(z) - \Aw_k(z) - s(z)$. By definition,
\beq
	\Aq_k(z) - \Aw_k(z) = \sum_{p, q}^{(k)} A_{kp} R^{(k)}_{pq}(z) A_{qk} + \frac{J}{N} \sum_{p, q}^{(k)} A_{kp} R^{(k)}_{pq}(z) + \frac{J}{N} \sum_{p, q}^{(k)} R^{(k)}_{pq}(z) A_{qk}.
\eeq
Using Lemma \ref{lem:M deviation} with $J=0$ (or the second part and the third part of Lemma \ref{lem:large deviation}), we find that
\beq
	\left| \sum_{p, q}^{(k)} A_{kp} R^{(k)}_{pq}(z) A_{qk} - \frac{1}{N} \sum_p^{(k)} R^{(k)}_{pp}(z) \right| \prec \frac{\|R^{(k)}\|}{\sqrt N}.
\eeq
Moreover, from the first part of Lemma \ref{lem:large deviation},
\beq
	\frac{1}{N} \left| \sum_{p, q}^{(k)} A_{kp} R^{(k)}_{pq}(z) \right| \prec \frac{1}{N} \sum_q^{(k)} \left( \frac{1}{N} \sum_p^{(k)} |R^{(k)}_{pq}(z)|^2 \right)^{\frac{1}{2}} \leq \left( \sum_q^{(k)} \frac{1}{N} \sum_p^{(k)} |R^{(k)}_{pq}(z)|^2 \right)^{\frac{1}{2}} = \frac{\|R^{(k)}\|}{\sqrt N}.
\eeq
Since $\| R^{(k)} \| \leq \frac{1}{\im z}$ and $z \in \caK$, and $|s^{(k)}(z) - s(z)| \prec N^{-1}$, we obtain that
\beq
	|\Aq_k(z) - \Aw_k(z) - s(z)| \prec N^{-\frac{1}{2}}.
\eeq

Now we consider \eqref{wrGamsumd}. 
We note that
\beq \label{G1}
	\E_k \left( \E_{k-1} [M_{kk}] \right)^2 = \frac{w_2}{N} + O(N^{-2})
\eeq
and
\beq \label{G2}
	\E_k \left[ \E_{k-1} \left[ M_{kk} \right] \cdot \E_{k-1} \left[ \Aq_k(z_2)-\Aw_k(z_2) - s_2 \right] \right] = \caO(N^{-\frac{3}{2}}) .
\eeq
We also have
\beq \begin{split} \label{G3}
	&\E_k \left[ \E_{k-1} \left[ \Aq_k(z_1)\right] \cdot \E_{k-1} \left[ \Aw_k(z_2) + s_2 \right] \right] \\
	&= \sum_{p, q}^{(k)} \E[ M_{kp} M_{qk}] \cdot \E_k \left[ \E_{k-1} [ R^{(k)}_{pq}(z_1) ] \cdot \E_{k-1} \left[\frac{J^2}{N^2} \sum_{r, t}^{(k)} R^{(k)}_{rt}(z_2) + s_2 \right] \right] \\
&= \E_k \left[ \E_{k-1} \left[\frac{J^2}{N^2} \sum_{p, q}^{(k)} R^{(k)}_{pq}(z_1) + s^{(k)}_N (z_1) \right] \cdot \E_{k-1} \left[\frac{J^2}{N^2} \sum_{r, t}^{(k)} R^{(k)}_{rt}(z_2) + s_2 \right] \right] \\
	&=\E_k \left[ \E_{k-1} \left[ \Aw_k(z_1) + s^{(k)}_N (z_1) \right] \cdot \E_{k-1} \left[ \Aw_k(z_2) + s_2 \right] \right].
\end{split} \eeq
Similar estimates hold if $z_2$ in \eqref{G2} and \eqref{G3} is replaced by $z_1$.
Noting the similarity of the formula of \eqref{G3} with $\E_k \left[ \E_{k-1} \left[ \Aw_k(z_1) + s_1 \right] \cdot \E_{k-1} \left[ \Aw_k(z_2) + s_2 \right] \right]$, \eqref{wrGamsumd} becomes
\beq \begin{split} \label{wrGamsumd1}
	&\frac{w_2}{N} + \E_k \left[ \E_{k-1} \left[ \Aq_k(z_1) \right] \cdot \E_{k-1} \left[ \Aq_k(z_2) \right] \right]
	- \E_k \left[ \E_{k-1} \left[ \Aw_k(z_1) + s^{(k)}_N (z_1) \right] \cdot \E_{k-1} \left[ \Aw_k(z_2) + s^{(k)}_N (z_2) \right] \right] \\
	\\ &+ \E_k \left[ \E_{k-1} \left[ s_1-s^{(k)}_N (z_1) \right] \cdot \E_{k-1} \left[ s_2-s^{(k)}_N (z_2) \right] \right]
	+ \caO(N^{-\frac{3}{2}}).
\end{split} \eeq

\subsubsection{$\E_k \left[ \E_{k-1} \left[ \Aq_k(z_1) \right] \cdot \E_{k-1} \left[ \Aq_k(z_2) \right] \right]$} \label{subsub:S_k}

We compute
\beq \begin{split}
	&\E_k \left[ \E_{k-1} \left[ \Aq_k(z_1) \right] \cdot \E_{k-1} \left[ \Aq_k(z_2) \right] \right] \\
&= \E_k \left[ \E_{k-1} \left[ \sum_{p, q}^{(k)} \left( A_{kp} + \frac{J}{N} \right) R^{(k)}_{pq}(z_1) \left( A_{qk} + \frac{J}{N} \right) \right] \cdot \E_{k-1} \left[ \sum_{r, t}^{(k)} \left( A_{kr} + \frac{J}{N} \right) R^{(k)}_{rt}(z_2) \left( A_{tk} + \frac{J}{N} \right) \right] \right].
\end{split} \eeq
We rearrange it in descending order of $J$ and calculate the conditional expectations.

\begin{enumerate}[1)]

\item For $J^4$-terms, we get
$$
	\frac{J^4}{N^4} \E_k \left[ \E_{k-1} \left[  \sum_{p, q}^{(k)} R^{(k)}_{pq}(z_1) \right] \cdot \E_{k-1} \left[ \sum_{r, t}^{(k)} R^{(k)}_{rt}(z_2) \right] \right]
	= \E_k \left[ \E_{k-1} \left[ \Aw_k(z_1) \right] \cdot \E_{k-1} \left[ \Aw_k(z_2) \right] \right].
$$

\item For $J^3$-terms, the conditional expectation vanishes because it always contains a factor $\E[A_{k \cdot}]$ or $\E[A_{\cdot k}]$.

\item For $J^2$-terms, we get
\begin{equation*} \begin{split}
	&\frac{J^2}{N^2} \E_k \left[ \E_{k-1} \left[ s^{(k)}_N (z_1) \right] \cdot \E_{k-1} \left[ \sum_{r, t}^{(k)} R^{(k)}_{rt}(z_2) \right] 
	+  \E_{k-1} \left[  \sum_{p, q}^{(k)} R^{(k)}_{pq}(z_1) \right] \cdot \E_{k-1} \left[ s^{(k)}_N (z_2) \right] \right] \\
	& = \E_k \left[ \E_{k-1} \left[ s^{(k)}_N (z_1) \right] \cdot \E_{k-1} \left[ \Aw_k(z_2) \right] 
	+  \E_{k-1} \left[  \Aw_k(z_1) \right] \cdot \E_{k-1} \left[ s^{(k)}_N (z_2) \right] \right] .
\end{split} \end{equation*}
We also have other terms, but they are all negligible, i.e., of order $\caO(N^{-\frac{3}{2}})$. (After summing over $k$, the contribution from such terms will be $N^{-\frac{1}{2}}$.) For example,  consider 
\beq \begin{split}
	X_k&=\frac{J^2}{N^2} \E_k \left[ \E_{k-1} \left[ \sum_{p, q}^{(k)} R^{(k)}_{pq}(z_1) A_{qk} \right] \cdot \E_{k-1} \left[ \sum_{r, t}^{(k)} A_{kr} R^{(k)}_{rt}(z_2) \right] \right] \\
&= \frac{J^2}{N^3} \E_k \left[ \sum_{q:q>k}  \E_{k-1}  \left[\sum_p^{(k)} R^{(k)}_{pq}(z_1) \right] \cdot \E_{k-1} \left[ \sum_t^{(k)} R^{(k)}_{qt}(z_2) \right] \right].
\end{split} \eeq
By naive power counting, we see that $X_k = \caO(N^{-1})$. 
Since the contribution from the case $p=q$ is $\caO(N^{-\frac{3}{2}})$, we may assume that $p \neq q$. Expanding $R^{(k)}_{pq}$ by
\beq
R^{(k)}_{pq}(z_1) = -R^{(k)}_{pp}(z_1) \sum_a^{(k, p)} M_{pa} R^{(k, p)}_{aq}(z_1) = -s_1 \sum_a^{(k, p)} M_{pa} R^{(k, p)}_{aq}(z_1) + \caO(N^{-1}),
\eeq
we obtain that
\beq \begin{split}
	X_k &= \frac{-J^2 s_1}{N^3} \E_k \left[ \sum_{q:q>k} \E_{k-1} \left[ \sum_p^{(k)} \sum_a^{(k, p)} M_{pa} R^{(k, p)}_{aq}(z_1) \right] \cdot \E_{k-1} \left[ \sum_t^{(k)} R^{(k)}_{qt}(z_2)\right]  \right] + \caO(N^{-\frac{3}{2}}) \\
&= \frac{-J^3 s_1}{N^4} \E_k \left[ \sum_{q:q>k} \E_{k-1} \left[ \sum_p^{(k)} \sum_a^{(k, p)} R^{(k, p)}_{aq}(z_1)\right]  \cdot \E_{k-1} \left[ \sum_t^{(k)} R^{(k)}_{qt}(z_2) \right] \right] + \caO(N^{-\frac{3}{2}}) \\
&= \frac{-J^3 s_1}{N^4} \E_k \left[ \sum_{q:q>k} \E_{k-1} \left[ \sum_p^{(k)} \sum_a^{(k)} R^{(k)}_{aq}(z_1) \right] \cdot \E_{k-1} \left[ \sum_t^{(k)} R^{(k)}_{qt}(z_2) \right] \right] + \caO(N^{-\frac{3}{2}}) \\
&= -Js_1 X_k + \caO(N^{-\frac{3}{2}}).
\end{split} \eeq
Hence, $X_k = \caO(N^{-\frac{3}{2}})$, which is negligible.

\item The $J$-terms can be computed as in the previous case and find that the contribution is negligible, i.e., $\caO(N^{-\frac{3}{2}})$. Since the computation is similar to the previous case, we skip the proof. 

\item For the terms with no $J$, the conditional expectation vanishes unless $|\{ p, q, r, t \}| = 2$ or $p=q=r=t$.

	\begin{enumerate}
\item If $p=q \neq r=t$, we get
\beq \begin{split}
&\frac{1}{N^2} \E_k \left[ \sum_{p \neq r}^{(k)} \E_{k-1} \left[ R^{(k)}_{pp}(z_1) \right] \cdot \E_{k-1} \left[ R^{(k)}_{rr}(z_2) \right] \right] \\
&= \E_k \left[ \E_{k-1} \left[s^{(k)}_N (z_1) \right] \cdot \E_{k-1} \left[s^{(k)}_N (z_2) \right] \right] - \frac{s_1 s_2}{N} + \caO(N^{-2}).
\end{split} \eeq

\item If $p=q=r=t$, we get
\beq \begin{split}
&\E_k \left[ \sum_p^{(k)} \E_{k-1} \left[ A_{kp} R^{(k)}_{pp}(z_1) A_{pk} \right] \cdot \E_{k-1} \left[ A_{kp} R^{(k)}_{pp}(z_2) A_{pk} \right] \right]  \\
&= \sum_{p:p<k} \frac{s_1 s_2}{N^2} + \sum_{p:p>k} \frac{W_4 s_1 s_2}{N^2} + \caO(N^{-2}) = \frac{k}{N} \frac{s_1 s_2}{N} + \frac{N-k}{N} \frac{W_4 s_1 s_2}{N} + \caO(N^{-2}).
\end{split} \eeq

\item If $p=t \neq q=r$, we get
\beq \begin{split}
&\E_k \left[ \sum_{p \neq q}^{(k)} \E_{k-1} \left[ A_{kp} R^{(k)}_{pq}(z_1) A_{qk} \right] \cdot \E_{k-1} \left[ A_{kq} R^{(k)}_{qp}(z_2) A_{pk} \right] \right] \\
	&= \E_k \left[ \sum_{p, q: p, q > k, p \neq q} (A_{kp} A_{qk} )^2 \cdot \E_{k-1} \left[ R^{(k)}_{pq}(z_1) \right] \cdot \E_{k-1} \left[ R^{(k)}_{qp}(z_2)  \right] \right] \\
&= \frac{1}{N^2} \E_k \left[ \sum_{p, q: p, q > k, p \neq q} \E_{k-1} \left[ R^{(k)}_{pq}(z_1) \right] \cdot \E_{k-1} \left[ R^{(k)}_{qp}(z_2) \right] \right] =: Y_k.
\end{split} \eeq
We note that $Y_k = \caO(N^{-1})$. The idea in the estimate for $Y_k$ is similar to that for $X_k$, except that we expand both $R^{(k)}_{pq}(z_1)$ and $R^{(k)}_{qp}(z_2)$. Then,
\beq \begin{split}
Y_k &= \frac{s_1 s_2}{N^2} \E_k \left[ \sum_{p, q: p, q > k, p \neq q} \sum_{a, b}^{(k, p)} \E_{k-1} \left[ M_{pa} R^{(k, p)}_{aq}(z_1) \right] \cdot \E_{k-1} \left[ R^{(k, p)}_{qb}(z_2) M_{bp} \right] \right] + \caO(N^{-\frac{3}{2}}) \\
&= \frac{s_1 s_2}{N^3} \E_k \left[ \sum_{p, q: p, q > k, p \neq q} \sum_{a:a>k}^{(p)} \E_{k-1} \left[ R^{(k, p)}_{aq}(z_1) \right] \cdot \E_{k-1} \left[ R^{(k, p)}_{qa}(z_2) \right] \right] + \caO(N^{-\frac{3}{2}}) \\
&= \frac{s_1 s_2}{N^3} \E_k \left[ \sum_{p, q: p, q > k, p \neq q} \sum_{a:a>k} \E_{k-1} \left[ R^{(k)}_{aq}(z_1) \right] \cdot \E_{k-1} \left[ R^{(k)}_{qa}(z_2) \right] \right] + \caO(N^{-\frac{3}{2}}) \\
&= \frac{N-k-1}{N} \frac{s_1 s_2}{N^2} \E_k \left[ \sum_{a, q: a, q > k} \E_{k-1} \left[ R^{(k)}_{aq}(z_1) \right] \cdot \E_{k-1} \left[ R^{(k)}_{qa}(z_2) \right] \right] + \caO(N^{-\frac{3}{2}}).
\end{split} \eeq
Thus, writing the last sum for $a\neq q$ and $a=q$ separately, we find that
\beq
Y_k = \frac{N-k-1}{N} s_1 s_2 Y_k + \frac{(N-k-1)(N-k)}{N^3} (s_1 s_2)^2 + \caO(N^{-\frac{3}{2}}),
\eeq
and we obtain that
\beq
Y_k = \left( 1 - \frac{N-k}{N} s_1 s_2 \right)^{-1} \frac{(N-k)^2}{N^3} (s_1 s_2)^2 + \caO(N^{-\frac{3}{2}}).
\eeq

\item If $p=r \neq q=t$, the conditional expectation is the same as $Y_k$ in the previous case. (When $A$ is complex Hermitian, it vanishes.)

	\end{enumerate}

\end{enumerate}

Altogether, we obtain that
\beq \begin{split} \label{G4}	
	&\E_k \left[ \E_{k-1} \left[ \Aq_k(z_1) \right] \cdot \E_{k-1} \left[ \Aq_k(z_2) \right] \right] = \E_k \left[ \E_{k-1} \left[ \Aw_k(z_1) + s^{(k)}_N (z_1) \right] \cdot \E_{k-1} \left[ \Aw_k(z_2) + s^{(k)}_N (z_2) \right] \right] \\ 
	&- \frac{s_1 s_2}{N} + \frac{k}{N} \frac{s_1 s_2}{N} + \frac{N-k}{N} \frac{W_4 s_1 s_2}{N}  + 2\left( 1 - \frac{N-k}{N} s_1 s_2 \right)^{-1} \frac{(N-k)^2}{N^3} (s_1 s_2)^2 + \caO(N^{-\frac{3}{2}}).
\end{split} \eeq

\subsubsection{Conclusion for $\wt \Gamma_N (z_1, z_2)$ and $\Gamma_N (z_1, z_2)$}

Combining \eqref{wrGamsumd1} and \eqref{G4}, we find that \eqref{wrGamsumd} is equal to 
\beq \begin{split} \label{G45}	
	\frac{w_2}{N}- \frac{s_1 s_2}{N} + \frac{k}{N} \frac{s_1 s_2}{N} + \frac{N-k}{N} \frac{W_4 s_1 s_2}{N}  + 2\left( 1 - \frac{N-k}{N} s_1 s_2 \right)^{-1} \frac{(N-k)^2}{N^3} (s_1 s_2)^2 + \caO(N^{-\frac{3}{2}}).
\end{split} \eeq
Summing over $k$, we obtain from \eqref{tilde Gamma} that 
\beq \begin{split}
	\wt \Gamma_N (z_1, z_2) 
	&= w_2 - 1 + \frac{1}{2} \left(W_4 - 3 \right) s_1 s_2 - \frac{2 \log (1- s_1 s_2)}{s_1 s_2} + \caO(N^{-\frac{1}{2}}),
\end{split} \eeq
where we used
\beq
\int_0^1 \frac{a^2 x^2}{1- ax} \dd x = -\frac{a}{2} - 1 - \frac{\log (1-a)}{a}
\eeq
for $a\in \C\setminus [1, \infty)$. 
To check that $s_1 s_2 \in \C\setminus [1, \infty)$, we notice that $\im s_1(z), \im s_2(z) > 0$ for $z \in \caK$. If $(\re s_1)(\re s_2) > 0$, $\im (s_1 s_2) \neq 0$. If $(\re s_1)(\re s_2) \leq 0$, $\re (s_1 s_2) < 0$. Thus, in any case, $s_1 s_2 \in \C\setminus [1, \infty)$.
Therefore,
\beq \begin{split} \label{eq:covariance gamma_N}
\Gamma_N (z_1, z_2) &= \frac{\partial^2}{\partial z_1 \partial z_2} \left[ s_1 s_2 \wt \Gamma_N (z_1, z_2) \right] \\
&= s_1' s_2' \left( (w_2 - 1) + 2 \left(W_4 - 3 \right) s_1 s_2 + \frac{2}{(1-s_1 s_2)^2} \right) + \caO(N^{-\frac{1}{2}}),
\end{split} \eeq
which converges to $\Gamma(z_1, z_2)$ in probability.

\section{Proof of Proposition \ref{prop:gaussian xi}} \label{sec:miscellanies}

We conclude the proof of Proposition \ref{prop:gaussian xi} 
by establishing the (a) the finite-dimensional convergence to Gaussian vectors and (b) the tightness of $\xi_N(z)$, as discussed in Section \ref{sec:outline}.

\subsection{Finite-dimensional convergence} \label{sub:martingale CLT}

To prove the finite-dimensional convergence, we use Theorem 35.12 of \cite{Billingsley_prob_meas} for Martingale central limit theorem. 
Recall the definition of $\phi_k$ in \eqref{phi_decompose} and \eqref{def_phi}.
Since we already proved the convergence of the variance in the previous section, it suffices to check that
\beq
\sum_{k=1}^N \E \left[ |\E_{k-1} [\phi_k]|^2 \chi_{|\E_{k-1} [\phi_k]| \geq \epsilon} \right] \to 0
\eeq
for any ($N$-independent) $\epsilon > 0$, as $N \to \infty$. Since
\beq
	\E \left[ |\E_{k-1} [\phi_k]|^2 \chi_{|\E_{k-1} [\phi_k]| \geq \epsilon} \right] 
	\leq \frac{1}{\epsilon^2} \E \left[ |\E_{k-1} [\phi_k]|^4 \right],
\eeq
it is sufficient to prove that
\beq \label{Lyapunov}
\sum_{k=1}^N \E \left[ |\E_{k-1} [\phi_k]|^4 \right] \to 0
\eeq
as $N \to 0$, which is the Lyapunov condition in \cite{BY2005}. The Lypanov condition \eqref{Lyapunov} is obvious from the estimate $\phi_k \prec N^{-\frac{1}{2}}$, which was established in the previous section.

\subsection{Tightness of $(\zeta_N)$} 
\label{sub:tightness}

{Since $\xi_N(z)=\zeta_N(z)+ \E[\xi_N(z)]$ and the mean $\E[\xi_N(z)]$ converges, it is enough to check the tightness of the sequence $\zeta_N(z)$.}
From Theorem 12.3 of \cite{Billingsley_conv}, it suffices to show that $(\zeta_N(z))$ is tight for a fixed $z$ and the following H\"older condition as in \cite{BY2005}: for some ($N$-independent) constant $K >0$,
\beq \label{eq:Holder condition}
	\E|\zeta_N(z_1) - \zeta_N(z_2)|^2 \leq K|z_1  - z_2|^2, \qquad z_1, z_2 \in \caK.
\eeq

The fact that $(\zeta_N(z))$ is tight for a fixed $z$ is obvious from that the variance is bounded uniformly on $N$ as shown in \eqref{eq:covariance gamma_N}. 

We now check the H\"older condition. Note that since 
$R(z_1) - R(z_2) = (z_1 - z_2) R(z_1) R(z_2)$,
we have
\beq \begin{split} \label{holder1}
	&\E|\zeta_N(z_1) - \zeta_N(z_2)|^2 
	= |z_1 - z_2|^2 \E | \Tr R(z_1) R(z_2) - \E \Tr R(z_1) R(z_2)|^2 \\
	&\qquad = |z_1 - z_2|^2 \E \left| \sum_{k=1}^N (\E_{k-1} - \E_k) \left( \Tr R(z_1) R(z_2) - \Tr R^{(k)}(z_1) R^{(k)}(z_2) \right) \right|^2.
\end{split} \eeq
We follow the arguments in Section \ref{sec:covariance} to estimate the right hand side of \eqref{holder1}. When compared with \eqref{eq:zeta_N decomposition0}, the main difference is that we do not need to precisely find the leading order term as in the covariance computation in Section \ref{sec:covariance}.

For the ease of notation, we set
\beq
R \equiv R(z_1), \qquad S \equiv R(z_2).
\eeq
We will frequently use the estimate
\beq
\| R \|, \| S \|, \| R^{(k)} \|, \| S^{(k)} \| \leq C.
\eeq
for any $k = 1, 2, \dots, N$, uniformly for $z_1, z_2 \in \caK$. For $i, j \neq k$,
\beq \begin{split}
R_{ij} S_{ji} - R^{(k)}_{ij} S^{(k)}_{ji} &= \left( R_{ij} - R^{(k)}_{ij} \right) S^{(k)}_{ji} + R^{(k)}_{ij} \left( S_{ji} - S^{(k)}_{ji} \right) + \left( R_{ij} - R^{(k)}_{ij} \right) \left( S_{ji} - S^{(k)}_{ji} \right) \\
&= \frac{R_{ik} R_{kj}}{R_{kk}} S^{(k)}_{ji} + R^{(k)}_{ij} \frac{S_{jk}{S_{ki}}}{S_{kk}} + \frac{R_{ik} R_{kj}}{R_{kk}} \frac{S_{jk}{S_{ki}}}{S_{kk}}.
\end{split} \eeq
Thus, using \eqref{RRdiff}, 
\beq
	\Tr RS - \Tr \left( R^{(k)} S^{(k)} \right) = \sum_{i, j}^{(k)} \left( \frac{R_{ik} R_{kj}}{R_{kk}} S^{(k)}_{ji} + R^{(k)}_{ij} \frac{S_{jk}{S_{ki}}}{S_{kk}} + \frac{R_{ik} R_{kj}}{R_{kk}} \frac{S_{jk}{S_{ki}}}{S_{kk}} \right) + 2(RS)_{kk}
\eeq
and
\beq \begin{split} \label{eq:rs difference}
	&\E|\zeta_N(z_1) - \zeta_N(z_2)|^2 \\
	&= |z_1 - z_2|^2 \E \left| \sum_{k=1}^N (\E_{k-1} - \E_k) \sum_{i, j}^{(k)} \left( \frac{R_{ik} R_{kj}}{R_{kk}} S^{(k)}_{ji} + R^{(k)}_{ij} \frac{S_{jk}{S_{ki}}}{S_{kk}} + \frac{R_{ik} R_{kj}}{R_{kk}} \frac{S_{jk}{S_{ki}}}{S_{kk}} \right) + 2(RS)_{kk} \right|^2 \\
	&= |z_1 - z_2|^2 \E \sum_{k=1}^N \left| (\E_{k-1} - \E_k) \sum_{i, j}^{(k)} \left( \frac{R_{ik} R_{kj}}{R_{kk}} S^{(k)}_{ji} + R^{(k)}_{ij} \frac{S_{jk}{S_{ki}}}{S_{kk}} + \frac{R_{ik} R_{kj}}{R_{kk}} \frac{S_{jk}{S_{ki}}}{S_{kk}} \right) + 2(RS)_{kk} \right|^2,
\end{split} \eeq
where we used \eqref{eq:E_k sum change} to get the last line. 

To estimate the right hand side of \eqref{eq:rs difference}, we rewrite the first term in the summand as
\beq
	\sum_{i, j}^{(k)} \frac{R_{ik} R_{kj}}{R_{kk}} S^{(k)}_{ji} = \sum_{i, j}^{(k)} R_{kk} \sum_{p, q}^{(k)} M_{pk} R^{(k)}_{ip} R^{(k)}_{qj} M_{kq} S^{(k)}_{ji} = R_{kk} \sum_{p, q}^{(k)} M_{kq} \left( R^{(k)} S^{(k)} R^{(k)} \right)_{qp} M_{pk}.
\eeq
Since
\beq
	(\E_{k-1} - \E_k) \left[ \frac{s_N^{(k)} (z_1)}{N} \sum_p \left( R^{(k)} S^{(k)} R^{(k)} \right)_{pp} \right] = 0,
\eeq
we obtain
\beq \begin{split} \label{rs1}
	&\E \sum_{k=1}^N \left| (\E_{k-1} - \E_k) \sum_{i, j}^{(k)} \frac{R_{ik} R_{kj}}{R_{kk}} S^{(k)}_{ji} \right|^2 \\
	&= \E \sum_{k=1}^N \left| (\E_{k-1} - \E_k) \left[ R_{kk} \sum_{p, q}^{(k)} M_{kq} \left( R^{(k)} S^{(k)} R^{(k)} \right)_{qp} M_{pk} - \frac{s_N^{(k)} (z_1)}{N} \sum_p \left( R^{(k)} S^{(k)} R^{(k)} \right)_{pp} \right] \right|^2 \\
	&\leq 4 \sum_{k=1}^N \E \left| R_{kk} \sum_{p, q}^{(k)} M_{kq} \left( R^{(k)} S^{(k)} R^{(k)} \right)_{qp} M_{pk} - \frac{s_N^{(k)} (z_1)}{N} \sum_p \left( R^{(k)} S^{(k)} R^{(k)} \right)_{pp} \right|^2.
\end{split} \eeq
Using that $|R_{kk}| \leq \| R \| \leq C$, we get
\beq \begin{split} \label{rs11}
	&\E \left| R_{kk} \sum_{p, q}^{(k)} M_{kq} \left( R^{(k)} S^{(k)} R^{(k)} \right)_{qp} M_{pk} - \frac{R_{kk}}{N} \sum_p \left( R^{(k)} S^{(k)} R^{(k)} \right)_{pp} \right|^2 \\
	&\leq C \, \E \left| \sum_{p, q}^{(k)} M_{kq} \left( R^{(k)} S^{(k)} R^{(k)} \right)_{qp} M_{pk} - \frac{1}{N} \sum_p \left( R^{(k)} S^{(k)} R^{(k)} \right)_{pp} \right|^2 \leq \frac{C \| R^{(k)} S^{(k)} R^{(k)}\|^2}{N} \leq \frac{C}{N},
\end{split} \eeq
where we used Lemma \ref{lem:M deviation} to get the second inequality. Moreover, since
\beq
	\left| \frac{1}{N} \sum_p^{(k)} \left( R^{(k)} S^{(k)} R^{(k)} \right)_{pp} \right| \leq \left\| R^{(k)} S^{(k)} R^{(k)} \right\| \leq C,
\eeq
we also have that
\beq \label{rs12'}
	\E \left| \frac{R_{kk}}{N} \sum_p \left( R^{(k)} S^{(k)} R^{(k)} \right)_{pp} - \frac{s_N^{(k)} (z_1)}{N} \sum_p \left( R^{(k)} S^{(k)} R^{(k)} \right)_{pp} \right|^2 \leq C \, \E \left| R_{kk} - s_N^{(k)} (z_1) \right|^2.
\eeq
Recall that we defined $Q_k= -M_{kk} + \sum_{p, q}^{(k)} M_{kp} R^{(k)}_{pq} M_{qk}$. Applying \eqref{schur expand} to expand $R_{kk}$ and using Corollary \ref{cor:local law}, we find that
\beq \label{r-s expansion}
	R_{kk} - s_N^{(k)} (z_1) = s(z_1) - s_N^{(k)} (z_1) + s(z_1)^2 (Q_k -s(z_1)) + \caO(N^{-1}) = s(z_1)^2 (Q_k - s_N^{(k)} (z_1) ) + \caO(N^{-1}).
\eeq
Thus, from Lemma \ref{lem:M deviation},
\beq \label{rs12''}
	\E \left| R_{kk} - s_N^{(k)} (z_1) \right|^2 \leq C \, \E \left| -M_{kk} + \sum_{p, q}^{(k)} M_{kp} R^{(k)}_{pq} M_{qk} - \frac{1}{N} \sum_p^{(k)} R^{(k)}_{pp} \right|^2 \leq \frac{C}{N}
\eeq
hence, together with \eqref{rs12'}, we get
\beq \label{rs12}
	\E \left| \frac{R_{kk}}{N} \sum_p \left( R^{(k)} S^{(k)} R^{(k)} \right)_{pp} - \frac{s_N^{(k)} (z_1)}{N} \sum_p \left( R^{(k)} S^{(k)} R^{(k)} \right)_{pp} \right|^2 \leq \frac{C}{N}.
\eeq

Combining \eqref{rs11} and \eqref{rs12} with \eqref{rs1}, we find that
\beq \label{eq:rs1 bound}
	\E \sum_{k=1}^N \left| (\E_{k-1} - \E_k) \sum_{i, j}^{(k)} \frac{R_{ik} R_{kj}}{R_{kk}} S^{(k)}_{ji} \right|^2 \leq C.
\eeq
Similarly, we can also obtain a bound
\beq \label{eq:rs2 bound}
	\E \sum_{k=1}^N \left| (\E_{k-1} - \E_k) \sum_{i, j}^{(k)} R^{(k)}_{ij} \frac{S_{jk}{S_{ki}}}{S_{kk}} \right|^2 \leq C. 
\eeq

We expand the third term of the summand in \eqref{eq:rs difference} as
\beq \begin{split}
	\sum_{i, j}^{(k)} \frac{R_{ik} R_{kj}}{R_{kk}} \frac{S_{jk}{S_{ki}}}{S_{kk}} 
	&= R_{kk} S_{kk} \sum_{i, j}^{(k)} \sum_{p, q}^{(k)} M_{pk} R^{(k)}_{ip} R^{(k)}_{qj} M_{kq} \sum_{r, t}^{(k)} M_{rk} S^{(k)}_{jr} S^{(k)}_{ti} M_{kt} \\
	&= R_{kk} S_{kk} \sum_{t, p}^{(k)} M_{kt} \left( S^{(k)} R^{(k)} \right)_{tp} M_{pk} \sum_{q, r}^{(k)} M_{kq} \left( R^{(k)} S^{(k)} \right)_{qr} M_{rk} \\
	&= R_{kk} S_{kk} \left( \sum_{p, q}^{(k)} M_{kp} \left( R^{(k)} S^{(k)} \right)_{pq} M_{qk} \right)^2
\end{split} \eeq
since $R$ and $S$ commute. Following the decomposition idea we used in the proof of \eqref{eq:rs1 bound}, we first observe
\beq
	(\E_{k-1} - \E_k) \left[ s_N^{(k)} (z_1) s_N^{(k)} (z_2) \left( \frac{1}{N} \sum_p \left( R^{(k)} S^{(k)} \right)_{pp} \right)^2 \right] = 0.
\eeq
Thus,
\beq \begin{split}
	&\E \left| (\E_{k-1} - \E_k) \sum_{i, j}^{(k)} \frac{R_{ik} R_{kj}}{R_{kk}} \frac{S_{jk}{S_{ki}}}{S_{kk}} R_{kk} \right|^2 \\
	&= \E \left| (\E_{k-1} - \E_k) \left[ R_{kk} S_{kk} \left( \sum_{p, q}^{(k)} M_{kp} \left( R^{(k)} S^{(k)} \right)_{pq} M_{qk} \right)^2 - s_N^{(k)} (z_1) s_N^{(k)} (z_2) \left( \frac{1}{N} \sum_p \left( R^{(k)} S^{(k)} \right)_{pp} \right)^2 \right] \right|^2.
\end{split} \eeq
Since $|R_{kk} S_{kk}| \leq \| R \| \| S \| \leq C$ and $\frac{1}{N} \sum_p \left( R^{(k)} S^{(k)} \right)_{pp} \leq \| R \| \| S \| \leq C$,
\beq \begin{split} \label{rs31}
	&\E \left| (\E_{k-1} - \E_k) \sum_{i, j}^{(k)} \frac{R_{ik} R_{kj}}{R_{kk}} \frac{S_{jk}{S_{ki}}}{S_{kk}} R_{kk} \right|^2 \\
	&\leq C \, \E \left| \left( \sum_{p, q}^{(k)} M_{kp} \left( R^{(k)} S^{(k)} \right)_{pq} M_{qk} \right)^2 - \left( \frac{1}{N} \sum_p \left( R^{(k)} S^{(k)} \right)_{pp} \right)^2 \right|^2 + C \, \E \left| R_{kk} S_{kk} - s_N^{(k)} (z_1) s_N^{(k)} (z_2) \right|^2. 
\end{split} \eeq
Using a simple identity $A^2 - B^2 = (A-B)^2 + 2B(A-B)$, we estimate the first term in the right hand side of \eqref{rs31} by
\beq \begin{split} \label{rs311}
	&\E \left| \left( \sum_{p, q}^{(k)} M_{kp} \left( R^{(k)} S^{(k)} \right)_{pq} M_{qk} \right)^2 - \left( \frac{1}{N} \sum_p \left( R^{(k)} S^{(k)} \right)_{pp} \right)^2 \right|^2 \\
	&= \E \left| \left( \sum_{p, q}^{(k)} M_{kp} \left( R^{(k)} S^{(k)} \right)_{pq} M_{qk} - \frac{1}{N} \sum_p \left( R^{(k)} S^{(k)} \right)_{pp} \right)^2 \right. \\
	&\qquad \qquad \qquad \left. + \frac{2}{N} \sum_p \left( R^{(k)} S^{(k)} \right)_{pp} \left( \sum_{p, q}^{(k)} M_{kp} \left( R^{(k)} S^{(k)} \right)_{pq} M_{qk} - \frac{1}{N} \sum_p \left( R^{(k)} S^{(k)} \right)_{pp} \right) \right|^2 \\
	&\leq C \, \E \left| \sum_{p, q}^{(k)} M_{kp} \left( R^{(k)} S^{(k)} \right)_{pq} M_{qk} - \frac{1}{N} \sum_p \left( R^{(k)} S^{(k)} \right)_{pp} \right|^2 \leq \frac{C}{N},
\end{split} \eeq
where we used Lemma \ref{lem:M deviation} in the last inequality. From \eqref{rs12}, we also find that
\beq \begin{split} \label{rs312}
	&\E \left| R_{kk} S_{kk} - s_N^{(k)} (z_1) s_N^{(k)} (z_2) \right|^2 = \E \left| \left( R_{kk} - s_N^{(k)} (z_1) \right) S_{kk} + s_N^{(k)} (z_1) \left( S_{kk} - s_N^{(k)} (z_2) \right) \right|^2 \\
	&\leq C \, \E \left| R_{kk} - s_N^{(k)} (z_1) \right|^2 + C \, \E \left| S_{kk} - s_N^{(k)} (z_2) \right|^2 \leq \frac{C}{N}.
\end{split} \eeq

From \eqref{rs31}, \eqref{rs311}, and \eqref{rs312}, we obtain a bound
\beq \label{eq:rs3 bound}
	\E \sum_{k=1}^N \left| (\E_{k-1} - \E_k) \sum_{i, j}^{(k)} \frac{R_{ik} R_{kj}}{R_{kk}} \frac{S_{jk}{S_{ki}}}{S_{kk}} \right|^2 \leq C.
\eeq

Finally, the last term in \eqref{eq:rs difference} becomes
\beq
	(RS)_{kk} = R_{kk} S_{kk} \sum_{p, q}^{(k)} M_{kq} \left( R^{(k)} S^{(k)} \right)_{qp} M_{pk},
\eeq
and one can prove by following the same argument as in the derivation of \eqref{eq:rs3 bound} that
\beq \label{eq:rs4 bound}
	\E \sum_{k=1}^N \left| (\E_{k-1} - \E_k) \left[ R_{kk} S_{kk} \sum_{p, q}^{(k)} M_{kq} \left( R^{(k)} S^{(k)} \right)_{qp} M_{pk} \right] \right|^2 \leq C.
\eeq

From \eqref{eq:rs difference}, \eqref{eq:rs1 bound}, \eqref{eq:rs2 bound}, \eqref{eq:rs3 bound}, and \eqref{eq:rs4 bound}, we find that the H\"older condition \eqref{eq:Holder condition} holds, which concludes the proof for tightness of $(\zeta_N)$.

\section{Proof of Lemma \ref{lem:Gamolrsm}} \label{sub:nonrandom}

For $z \in \Gamma_0$, we have $|s_N(z) - s(z)| \prec N^{-1}$ from Corollary \ref{cor:local law}. Thus, for any $\epsilon > 0$,
\beq
\int_{\Gamma_0} \E |\xi_N(z)|^2 \dd z \leq N^{-1 + \epsilon} |\Gamma_0| = 4 N^{-1 + \epsilon - \delta}.
\eeq
Setting $\epsilon = \frac{\delta}{2}$, we find that \eqref{nonrandom1} holds for $\Gamma_0$.

To prove \eqref{nonrandom1} for $\Gamma_r$, it suffices to show that $\E|\xi_N(z)|^2 < K$ for some ($N$-independent) constant $K>0$. In Section \ref{sec:mean}, we proved that
\beq
	\E \xi_N = \frac{s^2}{1-s^2} \left( -J' + \frac{J^2 s}{1 + Js} + (w_2-1)s + s's + \left( W_4 - 3 \right) s^3 \right) + O(N^{-\frac{1}{2} + \epsilon}),
\eeq
thus $|\E \xi_N|^2 < C$ for $z \in \Gamma_r$.

We now estimate $\E |\zeta_N|^2 = \E |\xi_N - \E \xi_N|^2$. Recall that we showed in \eqref{eq:zeta_N decomposition} that
\beq
	\zeta_N
	= \sum_{k=1}^N (\E_{k-1} - \E_k) \left[ R_{kk} \left( 1 + \sum_{p, q}^{(k)} M_{kp} (R^{(k)})^2_{pq} M_{qk} \right) \right].
\eeq
Following the idea in \eqref{rs11}, we use
\beq
	(\E_{k-1} - \E_k) \left[ s_N^{(k)} \left( 1 + \frac{1}{N} \sum_p^{(k)} (R^{(k)})^2_{pp}\right) \right] =0,
\eeq
hence
\beq \begin{split} \label{rl}
	\E |\zeta_N|^2 &= \E \sum_{k=1}^N \left| (\E_{k-1} - \E_k) \left[ R_{kk} \left( 1 + \sum_{p, q}^{(k)} M_{kp} (R^{(k)})^2_{pq} M_{qk} \right) - s_N^{(k)} \left( 1 + \frac{1}{N} \sum_p^{(k)} (R^{(k)})^2_{pp}\right) \right] \right|^2 \\
	&\leq 4 \sum_{k=1}^N \E \left| R_{kk} \left( 1 + \sum_{p, q}^{(k)} M_{kp} (R^{(k)})^2_{pq} M_{qk} \right) - s_N^{(k)} \left( 1 + \frac{1}{N} \sum_p^{(k)} (R^{(k)})^2_{pp}\right) \right|^2.
\end{split} \eeq

Define the event
\beq
	\Omega_N := \{ \mu_1 \leq \widehat{J} + N^{-1/3} \}. 
\eeq
From Lemma \ref{lem:largest eigenvalue}, we find $\p (\Omega_N) < N^{-D}$ for any (large) fixed $D>0$. On $\Omega_N$,
\beq
	|R_{kk}| \leq \| R \| \leq \frac{1}{a_+ -\widehat{J} - N^{-1/3}} \leq C
\eeq
for any $k=1, 2, \dots, N$, uniformly for $z \in \Gamma_r$. Similarly, $\| R^{(k)} \| \leq C$ for any $k=1, 2, \dots, N$, uniformly for $z \in \Gamma_r$. Thus,
\beq \begin{split} \label{rl1}
	&\E \left| \lone(\Omega_N) \left[ R_{kk} \left( 1 + \sum_{p, q}^{(k)} M_{kp} (R^{(k)})^2_{pq} M_{qk} \right) - R_{kk} \left( 1 + \frac{1}{N} \sum_p^{(k)} (R^{(k)})^2_{pp} \right) \right] \right|^2 \\
	&\leq C \, \E \left| \sum_{p, q}^{(k)} M_{kp} (R^{(k)})^2_{pq} M_{qk} - \frac{1}{N} \sum_p^{(k)} (R^{(k)})^2_{pp} \right|^2 \leq \frac{C \| R^{(k)} \|^4}{N} \leq \frac{C}{N},
\end{split} \eeq
Moreover, since
\beq
	\left| \frac{1}{N} \sum_p^{(k)} (R^{(k)})^2_{pp} \right| \leq \| R^{(k)} \|^2 \leq C
\eeq
on $\Omega_N$, from \eqref{rs12''}, we get
\beq \begin{split} \label{rl2}
	&\E \left| \lone(\Omega_N) \left[ R_{kk} \left( 1 + \frac{1}{N} \sum_p^{(k)} (R^{(k)})^2_{pp} \right) - s_N^{(k)} \left( 1 + \frac{1}{N} \sum_p^{(k)} (R^{(k)})^2_{pp}\right) \right] \right|^2 \\
	&\leq C \, \E \left| \lone(\Omega_N) \left[R_{kk} - s_N^{(k)} \right] \right|^2 \leq \frac{C}{N}.
\end{split} \eeq

On $\Omega_N^c$, we use the trivial bound $\| R \|, \| R^{(k)} \| \leq \frac{1}{\im z} \leq N^{\delta}$. Then,
\beq \begin{split} \label{rl1'}
	&\E \left| \lone(\Omega_N^c) \left[ R_{kk} \left( 1 + \sum_{p, q}^{(k)} M_{kp} (R^{(k)})^2_{pq} M_{qk} \right) - R_{kk} \left( 1 + \frac{1}{N} \sum_p^{(k)} (R^{(k)})^2_{pp} \right) \right] \right|^2 \\
	&\leq \left( \E \left[ \lone(\Omega_N^c) |R_{kk}|^2 \right] \right)^{1/2} \left( \E \left| \sum_{p, q}^{(k)} M_{kp} (R^{(k)})^2_{pq} M_{qk} - \frac{1}{N} \sum_p^{(k)} (R^{(k)})^2_{pp} \right|^4 \right)^{1/2} \leq C\p (\Omega_N^c)^{1/2} \frac{\| R \| \| R^{(k)} \|^4}{N} \leq \frac{C}{N}
\end{split} \eeq
and similarly,
\beq \label{rl2'}
	\E \left| \lone(\Omega_N^c) \left[ R_{kk} \left( 1 + \frac{1}{N} \sum_p^{(k)} (R^{(k)})^2_{pp} \right) - s_N^{(k)} \left( 1 + \frac{1}{N} \sum_p^{(k)} (R^{(k)})^2_{pp}\right) \right] \right|^2  \leq \frac{C}{N}.
\eeq

Combining \eqref{rl1}, \eqref{rl2}, \eqref{rl1'}, and \eqref{rl2'}, we obtain
\beq
	\E \left| R_{kk} \left( 1 + \sum_{p, q}^{(k)} M_{kp} (R^{(k)})^2_{pq} M_{qk} \right) - s_N^{(k)} \left( 1 + \frac{1}{N} \sum_p^{(k)} (R^{(k)})^2_{pp}\right) \right|^2 \leq \frac{C}{N},
\eeq
thus, from \eqref{rl},
\beq
	\E |\xi_N|^2 \leq 2\, \E |\zeta_N|^2 + 2\, |\E \xi_N|^2 \leq C,
\eeq
which proves the lemma for $\Gamma_r$. The proof of the lemma for $\Gamma_l$ is the same.

\section{Large deviation estimates} \label{subsec:largdeves}

We prove Lemma \ref{lem:M deviation} for the spiked random matrix $M$. 
The case of non-spiked random matrix is well-known and we adapt its proof. 
We use the following lemma, sometimes referred to as `large deviation estimates'.

\begin{lem}[Lemma 8.1 and Lemma 8.2 of \cite{EYY_bulk}] \label{lem:large deviation}
Let $a_1, \dots, a_N$ be independent (complex) random variables with mean zero and variance $1$. Suppose that $a_1, \dots, a_N$ satisfies the uniform subexponential decay condition. Then, for any deterministic complex numbers $A_i$ and $B_{ij}$ $(i, j = 1, 2, \dots, N)$,
\beq \begin{split}
\left| \sum_{i=1}^N A_i a_i \right| &\prec \left( \sum_{i=1}^N |A_i|^2 \right)^{\frac{1}{2}} \\
\left| \sum_{i=1}^N A_i |a_i|^2 - \sum_{i=1}^N A_i \right| &\prec \left( \sum_{i=1}^N |A_i|^2 \right)^{\frac{1}{2}} \\
\left| \sum_{i \neq j} a_i B_{ij} a_j \right| &\prec \left( \sum_{i \neq j} |B_{ij}|^2 \right)^{\frac{1}{2}}
\end{split} \eeq
\end{lem}

For the proof of Lemma \ref{lem:large deviation}, see Appendix B of \cite{EYY_bulk}.

\begin{proof}[Proof of Lemma \ref{lem:M deviation}]
We consider the case $n=1$ for the first part of the lemma. We first decompose
\beq \label{eq:M decompose}
M_{ip} S_{pq} M_{qi} = A_{ip} S_{pq} A_{qi} + \frac{J}{N} S_{pq} A_{qi} + \frac{J}{N} A_{ip} S_{pq} + \frac{J^2}{N^2} S_{pq}.
\eeq
Then,
\beq \begin{split} \label{eq:M variance}
&\left| \sum_{p, q}^{(i)} M_{ip} S_{pq} M_{qi} - \frac{1}{N} \sum_p^{(i)} S_{pp} \right|^2 \\
&\leq 4 \left| \sum_{p, q}^{(i)} A_{ip} S_{pq} A_{qi} - \frac{1}{N} \sum_p^{(i)} S_{pp} \right|^2 + \frac{4J^2}{N^2} \left| \sum_{p, q}^{(i)} S_{pq} A_{qi} \right|^2 + \frac{4J^2}{N^2} \left| \sum_{p, q}^{(i)} A_{ip} S_{pq} \right|^2 + \frac{4J^4}{N^4} \left| \sum_{p, q}^{(i)} S_{pq} \right|^2.
\end{split} \eeq
Taking the expectation,
\beq \begin{split}
&\E \left| \sum_{p, q}^{(i)} A_{ip} S_{pq} A_{qi} - \frac{1}{N} \sum_p^{(i)} S_{pp} \right|^2 \\
&= \E \sum_{p, q, r, s}^{(i)} A_{ip} S_{pq} A_{qi} A_{ir} \overline{S_{rs}} A_{si} - \E \sum_{p, q, r} A_{ip} S_{pq} A_{qi} \overline{S_{rr}} - \E \sum_{p, q, r} A_{ip} \overline{S_{pq}} A_{qi} S_{rr} + \frac{1}{N^2} \sum_{p, q}^{(i)} \overline{S_{pp}} S_{qq} \\
&= \frac{1}{N^2} \sum_{p, q}^{(i)} |S_{pq}|^2 + \frac{1}{N^2} \sum_{p, q}^{(i)} S_{pq} \overline{S_{qp}} + \frac{W_4}{N^2} \sum_p^{(i)} |S_{pp}|^2.
\end{split} \eeq
Since
\beq
\sum_{p, q}^{(i)} |S_{pq}|^2 = \| S \|_{HS}^2 \leq N \| S \|^2
\eeq
where $\| \cdot \|_{HS}$ denotes the Hilbert-Schmidt norm. Thus, we find that
\beq
\E \left| \sum_{p, q}^{(i)} A_{ip} S_{pq} A_{qi} - \frac{1}{N} \sum_p^{(i)} S_{pp} \right|^2 \leq \frac{W_4 +2}{N} \| S \|^2.
\eeq
Similarly, for other terms in \eqref{eq:M variance},
\beq
\left| \sum_{p, q}^{(i)} S_{pq} A_{qi} \right|^2 = \left| \sum_{p, q}^{(i)} A_{ip} S_{pq} \right|^2 \leq N \| S \|^2
\eeq
and
\beq
\left| \sum_{p, q}^{(i)} S_{pq} \right|^2 \leq N^2 \sum_{p, q}^{(i)} |S_{pq}|^2 \leq N^3 \| S \|^2.
\eeq
Altogether, we obtain that
\beq
\left| \sum_{p, q}^{(i)} M_{ip} S_{pq} M_{qi} - \frac{1}{N} \sum_p^{(i)} S_{pp} \right|^2 \leq 4(W_4 + 2 + 2J^2 + J^4) \frac{\| S \|^2}{N},
\eeq
which proves the first part of the lemma for $n=1$. The case $n=2$ can be proved analogously.

Next, we prove the second part of the lemma. From the second inequality in Lemma \ref{lem:large deviation},
\beq
\left| \sum_p^{(i)} A_{ip} S_{pp} A_{pi} - \frac{1}{N} \sum_p^{(i)} S_{pp} \right| \prec \frac{1}{N} \left( \sum_p^{(i)} |S_{pp}|^2 \right)^{\frac{1}{2}}.
\eeq
From the third inequality in Lemma \ref{lem:large deviation},
\beq
\left| \sum_{p \neq q}^{(i)} A_{ip} S_{pp} A_{pi} \right| \prec \frac{1}{N} \left( \sum_{p \neq q}^{(i)} |S_{pq}|^2 \right)^{\frac{1}{2}}.
\eeq
Summing the inequalities above, we find that
\beq
\left| \sum_{p, q}^{(i)} A_{ip} S_{pq} A_{qi} - \frac{1}{N} \sum_p^{(i)} S_{pp} \right| \prec \frac{1}{N} \left( \sum_{p, q}^{(i)} |S_{pq}|^2 \right)^{\frac{1}{2}} = \frac{\| S \|_{HS}}{N} \leq \frac{\| S \|}{\sqrt N}.
\eeq
For the second term in \eqref{eq:M decompose}, we apply the first inequality in Lemma \ref{lem:large deviation} and get
\beq \begin{split}
\left| \frac{J}{N} \sum_{p, q}^{(i)} S_{pq} A_{qi} \right| \prec \frac{1}{N\sqrt N} \sum_p^{(i)} \left( \sum_q^{(i)} |S_{pq}|^2 \right)^{\frac{1}{2}} \leq \frac{1}{N} \left( \sum_{p, q}^{(i)} |S_{pq}|^2 \right)^{\frac{1}{2}} \leq \frac{\| S \|}{\sqrt N}.
\end{split} \eeq
The same estimate holds for the third term in \eqref{eq:M decompose}. Finally, for the last term in \eqref{eq:M decompose},
\beq
\left| \sum_{p, q}^{(i)} \frac{J^2}{N^2} S_{pq} \right| \leq \frac{J^2}{N} \left( \sum_{p, q}^{(i)} |S_{pq}|^2 \right)^{\frac{1}{2}} \leq \frac{\| S \|}{\sqrt N}.
\eeq
Summing the estimates, we obtain \eqref{eq:large general}.
\end{proof}


\def\cydot{\leavevmode\raise.4ex\hbox{.}}

\end{document}